\documentclass{article}

\usepackage{a4wide,amsfonts,amssymb,amsmath,times,amsthm, dsfont,amsthm}
\usepackage{graphicx}
\usepackage{comment}

\title{2-Xor revisited: satisfiability and probabilities of functions}

\author{\'Elie de Panafieu\thanks{COMPLEX NETWORKS, University of Paris 6. Supported by the ANR projects
BOOLE (2009-13) and MAGNUM (2010-14), the P.H.C. Amadeus project (2013-14), the PEPS HYDrATA and the Austrian Science Fund (FWF) grant F5004.} 
\and 
Dani\`ele Gardy\thanks{DAVID Laboratory, University of Versailles Saint Quentin en Yvelines. Part of the work of this author was done during a long-term visit at the Institute of Discrete Mathematics and Geometry of the TU Wien. Supported by the P.H.C. Amadeus project {\em Probabilities and tree representations for Boolean functions} (2013-14) and by the ANR project BOOLE (2009-13).} 
\and 
Bernhard Gittenberger\thanks{Institute of Discrete Mathematics and Geometry, TU Wien. Supported by the FWF (Austrian Science Foundation), Special Research Program F50, grant F5003-N15, and by the \"OAD, grant Amad\'ee F01/2015.} 
\and
Markus Kuba\thanks{Supported by \"OAD, grant Amad\'ee F01/2015.}}

\begin{document}

\newtheorem{coroll}{Corollary}
\newtheorem{definition}{Definition}
\newtheorem{example}{Example}
\newtheorem{fact}{Fact}
\newtheorem{lemma}{Lemma}
\newtheorem{notation}{Notation}
\newtheorem{proposition}{Proposition}
\newtheorem{theorem}{Theorem}

\newcommand{\MP}[1]{\marginpar{\footnotesize #1}}

\theoremstyle{remark}

\newtheorem{remark}{Remark}

\newcommand{\EF}{{\cal EF}}
\newcommand{\CC}{{\cal FC}_2}
\newcommand{\classe}{{\cal C}}
\newcommand{\BB}{{\cal B}_n}
\newcommand{\f}{|E_f|}
\newcommand{\N}{E_{m,n}}
\newcommand{\M}{M_{m,n}}
\newcommand{\PP}{\vskip .5cm \noindent}
\newcommand{\true}{\textsc{True}}
\newcommand{\false}{\textsc{False}}
\newcommand{\ii}{\mathbf i}
\newcommand{\jj}{\mathbf j}
\newcommand{\kk}{\mathbf k}
\newcommand{\pp}{\mathcal P}
\newcommand{\xx}{\mathcal X}
\newcommand{\uu}{\mathbf u}
\newcommand{\vv}{\mathbf v}
\newcommand{\XX}{\mathbf X}

\newcommand{\proba}{\mathrm{Pr}_{[m,n]}}

\def\P{{\mathbb {P}}}

\newcommand{\bigO}{\mathcal{O}}
\newcommand{\bigOsoft}{\tilde{\mathcal{O}}}
\newcommand{\smallo}{o}
\newcommand{\prob}{\mathds{P}}
\newcommand{\MG}{M}
\newcommand{\vect}[1]{\mathbf{#1}}
\newcommand{\mg}{\mathcal M}
\newcommand{\core}{\operatorname{Core}}
\newcommand{\corenocycle}{\core^{(\setminus \text{cycle})}}
\newcommand{\seqv}{\operatorname{seqv}}


\newcommand{\mghat}{\hat{\MG}}
\renewcommand{\(}{\left(}
\renewcommand{\)}{\right)}
\newcommand{\Ptwoblocks}{\proba(2\; {\mathrm {blocks}})}

\newcommand{\fallfak}[2]{\ensuremath{#1^{\underline{#2}}}}

\bibstyle{plain}
\maketitle

\begin{abstract}
  The problem 2-Xor-Sat asks for the probability that a random expression, built as a conjunction of clauses $x \oplus y$, is satisfiable. 
  We revisit this classical problem by giving an alternative, explicit expression of this probability.
  We then consider a refinement of it, namely the probability that a random expression computes a specific Boolean function. 
The answers to both problems involve a description of 2-Xor expressions as multigraphs and use classical methods of analytic combinatorics by expressing probabilities through coefficients of generating functions.

\bigskip
{\bf Keywords: multigraph enumeration, probability of Boolean functions, satisfiability, 2-Xor expressions, asymptotics.}
\end{abstract}

\section{Introduction}

In constraint satisfaction problems we ask for the probability that a random expression, built on a finite set of Boolean variables according to some rules ($k$-Sat, $k$-Xor-Sat, NAE, \dots), is (un)satisfiable.
The behaviour of this probability, when the number $n$ of Boolean variables and the length $m$ of the expression (usually defined as the number of clauses) tend to infinity, has given rise to numerous studies, most of them concentrating on the existence and location of a threshold from satisfiability to unsatisfiability as the ratio $m/n$ grows.
The literature in this direction is vast; for Xor-functions see e.g. \cite{CreignouDaude99,CrDa03,CDD03,CrDa04,CREIGNOU-DAUDE-EGLY}.

Defining a probability distribution on Boolean functions through a distribution on Boolean expressions is \emph{a priori} a different question.
Quantitative logic aims at answering such a question, and many results have been obtained when the
Boolean expression, or equivalently the random tree that models it, is a variation of well-known
combinatorial or probabilistic tree models such as Galton-Watson and P\'olya trees, binary search
trees, etc (\cite{LS97, CFGG04, BrPi05, Wo05, zaionc2005, FGGZ07, GKZ08, Ko08, CGM11, GGKM11,
FGGG12, GGKM12}).

So we have two frameworks: On the one hand we try to determine the probability that an expression
is satisfiable; on the other hand we try to identify probability distributions on the set of 
Boolean functions. 
It is only natural that we should wish to merge these two approaches: We set satisfiability
problems into the framework of quantitative logic (this only requires choosing a suitable model of
expressions) and ask for the probability of $\false$ -- this is the classical satisfiability
problem  -- \emph{and} for the probabilities of the other Boolean functions as well. 
This amounts to refining the satisfiable case and taking all the functions different from $\false$
also into account. The set of Boolean expressions is then partitioned into subsets according to the 
(class of) Boolean function(s) that is computed. 

Within this unified framework one could, e.g., ask for the probability that a random expression
computes a function that is satisfied by a specific number of assigments.
Although this may turn out to be out of our reach for most classical satisfiability problems,
there are some problems for which we may still hope to obtain a (partial) description of the
probability distribution on the set of Boolean functions.
The case of 2-Xor expressions is such a problem, and this paper is devoted to presenting our
results in this domain.

Consider random 2-Xor-Sat instances with a large number~$n$ of variables, and~$m$ of clauses.
Creignou and Daud\'e established that their limit probability of satisfiability
goes from positive values to zero when the ratio~$m/n$ crosses~$1/2$ (see \cite{CreignouDaude99}).
They then proved that this threshold is coarse (\emph{cf.} \cite{CrDa04}).
Further work by Daud\'e and Ravelomanana \cite{DaudeRavelomanana} and by Pittel and Yeum
\cite{PittelYeum} led to a precise understanding of the transition in a window of size $n^{-1/3}$ around~$1/2$.

The paper is organized as follows. 
We present in the next section 2-Xor expressions and the set of Boolean functions that they can
represent. Then we give a modelization of these expressions in terms of multigraphs, before considering in Section~\ref{sec:probas} how enumeration results on classes of multigraphs allow us to compute probabilities of Boolean functions.
We then give explicit results for several classes of functions in Section~\ref{sec:results}, and conclude with a discussion on the relevance and of possible extensions of our work in Section~5.

A preliminary version of our work was presented at the conference Latin'14~\cite{dePGGK14}.

\section{Boolean Expressions and Functions and their Relations to Multigraphs}

\subsection{2-Xor Expressions and Boolean Functions}
\label{sec:model-expressions}

In this section we will lay out the framework of Boolean expressions which we will investigate. If
$x$ is a Boolean variable, we will denote by $\bar x$ its negation. 

\begin{definition}
Let $\{ x_1, x_2, \ldots, x_n \}$ be a set of Boolean variables. A 2-Xor expression is a finite
conjunction of clauses $l \oplus l'$, where $l$ and $l'$ are literals, i.e. they are elements of
$\{x_1, x_2, \ldots, x_n, \bar x_1, \bar x_2,\dots, \bar x_n\}$. 

The clauses as well as the literals within each clause are ordered (i.e. for instance that the
clauses $x\oplus y$ and $y \oplus x$ are distinct). 
From a combinatorial point of view, an expression can be regarded as a
\emph{sequence} of clauses where each clause is a pair of two literals. 
Neither the literals of a clause nor the clauses themselves need to be distinct. 

The set of all such expressions is denoted by $\mathcal E_n$.
\end{definition}

We say that a 2-Xor expression \emph{defines}, or \emph{computes},
the corresonding Boolean function.
We shall denote the number of clauses of an expression by $m$.
Now each 2-Xor expression defines a Boolean function on a finite number of variables, but not all
Boolean functions on a finite number of variables can be represented by a 2-Xor expression. 
We define~$\xx$ as the set of functions from $\{ 0,1 \}^{\mathbb N}$ to $\{ 0, 1 \}$ which have at
least one representation by a 2-Xor expression in $\bigcup_{n\ge 1} \mathcal E_n$.
We also define, for each $n \geq 1$, the set $\xx_n$ of functions in $\xx$ such that there exists
an expression in $\mathcal E_n$ representing the function.
This implies that $\xx_{n_1} \subset \xx_{n_2}$ for $n_1 \leq n_2$, and that $\xx = \cup_{n \geq 1} \xx_n$.\footnote{
For the sake of brevity, in the sequel ``(the set of) Boolean functions'' is to be understood as
either the set $\xx_n$ or the set $ \xx$, according to the context.}

Consider now the expressions in $\mathcal E_n$. There there are $4 n^2$ distinct clauses.
We assume that the $m$ clauses are drawn with a uniform probability (and hence with replacement).
This framework allows us to define, for each~$m$, a probability distribution on the set~$\xx_n$: 

\begin{definition}
Let $\N = (4 n^2)^m$ be the total number of expressions with $m$ clauses on the variables
$x_1$, \ldots, $x_n$, and $\N(f)$ denote the number of these expressions that compute~$f$.
Then, for a Boolean function $f \in \xx_n$ we set
$
\proba (f) = \frac{\N(f)}{\N}.
$
\end{definition}

\subsection{The Sets $\xx_n$}

Rewriting a clause $l_1 \oplus l_2$ as $l_1 \sim \bar l_2$ or $\bar l_1 \sim l_2$ (i.e., the literals $l_1$ and $l_2$
must take opposite values for the clause to evaluate to $\true$),  
and merging the clauses sharing a common variable, we see that the functions we
obtain can be written as a conjunction of equivalence relations on literals:
\footnote{Note that
the relation $\sim$ corresponds to an equivalence relation on the set of variables and therefore
induces a partition on the set of variables. But as to the presence of negations, the formal 
structure is in fact a little bit richer than only a set with an equivalence relation.}
$$
(l_{1} \sim \cdots \sim l_{p_1}) \wedge (l_{p_1+1} \sim \cdots \sim l_{p_2}) \wedge \cdots
\wedge (l_{p_{r-1}+1} \sim \cdots \sim l_{p_r}).
$$
E.g., for $n=7$ the expression 
$(x_1 \oplus x_3) \wedge (\bar x_6 \oplus x_5) \wedge (x_7 \oplus \bar{x}_7) \wedge (x_2 \oplus \bar{x}_3)$ 
computes a Boolean function $f$ that we can write as 
$(x_1 \sim \bar x_3) \wedge (x_6 \sim x_5) \wedge (x_7 \sim x_7) \wedge (\bar x_2 \sim \bar x_3)$ , 
or equivalently as 
$(x_1 \sim \bar x_2 \sim \bar x_3)   \wedge (x_5 \sim x_6)$; 
furthermore this function partitions the set of Boolean variables 
$\{ x_1, \ldots , x_7 \}$ into the subsets 
$\{x_1, x_2, x_3\}$, $\{ x_4 \}$, $\{x_5,x_6 \}$ and $\{ x_7 \}$.

If a clause inducing $l \sim \bar{l}$ appears, then the expression simply computes $\false$. In
other words:

\begin{proposition} \label{th:block_representation}
For any $n \geq 1$, the set $\xx_n$ of Boolean functions on $n$ variables, such that there exists
at least one 2-Xor expression in $\mathcal E_n$ that computes the function, comprises exactly the
function $\false$ and those functions $f$ that are specified as follows: 
Fix a set $Y=\{y_1,y_2,\dots,y_n\}$ such that $y_i=x_i$ or $y_i=\bar x_i$, for all $i=1,\dots,n$, and 
partition the set $Y$ into subsets. Then $f$ attains the value $\true$ if and only if for each
block of the partition all the literals have the same value. 
A variable which appears in no clause of an expression computing the function, 
or only as $l \sim l$, is put into a singleton. 
\end{proposition}

\begin{proof}
Given a set of literals~$p = \{l_1, \ldots, l_s\}$,
let~$\bar p$ denote the set where each literal is switched
\[
    \bar p = \{ \bar l_1, \ldots, \bar l_s \}.
\]
Let us first observe that if a satisfiable expression is specified, 
in the sens of the proposition, by the partition
\[
    Y = p_1 \uplus p_2 \uplus \cdots \uplus p_t,
\]
where each variable appears in exactly one literal of~$Y$,
then it is also specified by the partition
where any number of~$p_i$ is replaced by~$\bar p_i$.

We prove the proposition by recurrence on the number of clauses~$m$.
For~$m=0$, the Boolean function computed is $\true$,
and is specified by the partition
\[
    \{\{x_1\}, \{x_2\}, \ldots, \{x_n\}\}.
\]
of~$Y = \{x_1, \ldots, x_n\}$.
Let us assume that the proposition is proven for a given~$m$,
and consider a 2-Xor expression with~$m+1$ clauses
\[
    E = \tilde{E} \wedge (l_1 \oplus l_2),
\]
where $\tilde{E}$ is a 2-Xor expression with~$m$ clauses.
If $\tilde{E}$ computes the Boolean function $\false$,
then~$E$ also computes $\false$ and the proposition holds.
Otherwise, let
\[
    Y = p_1 \uplus p_2 \uplus \cdots \uplus p_t
\] 
denote the partition obtained by application of the proposition to the expression~$\tilde{E}$.
The last clause of~$E$ is $(l_1 \oplus l_2)$,
which is equivalent with $l_1 \sim \bar l_2$
and is satisfied if and only if $l_1$ and $\bar l_2$
are assigned the same Boolean value.
Without loss of generality, we can assume
that~$l_1$ belongs to~$Y$.
Otherwise, we just replace the set $p_i$ from the partition
that contains~$\bar l_1$ with $\bar p_i$.
\begin{itemize}
\item
If~$l_2$ also belongs to~$p_i$ then, according to the proposition,
$\tilde E$ is satisfied only if~$l_1$ and~$l_2$ take the same Boolean value,
so the clause~$(l_1 \oplus l_2)$ cannot be satisfied.
Therefore, $E$ is not satisfiable, so it computes the Boolean function $\false$.
\item
If~$\bar l_2$ belongs to~$p_i$, then the clause $l_1 \oplus l_2$
is satisfied by any assignment satisfying~$\tilde{E}$,
so~$E$ is satisfiable, and the partition built by the proposition for~$E$ is $Y = p_1 \uplus p_2 \uplus \cdots \uplus p_t$.
\item
Otherwise, there is a set~$p_j$ from~$P$, distinct from~$p_i$, that contains either~$l_2$ or~$\bar l_2$.
Without loss of generality, we can assume that~$p_j$ contains~$\bar l_2$.
Otherwise, we simply replace~$p_j$ with $\bar p_j$.
Then $E$ is satisfiable.
The corresponding partition is obtained from~$(p_1, \ldots, p_t)$
by replacing the sets~$p_i$ and~$p_j$ with $p_i \cup p_j$. \qedhere
\end{itemize}
\end{proof}

We now define an equivalence relation on $\xx_n$.
\begin{definition}
Two Boolean functions $f$ and $g$ on $n$ variables are \emph{equivalent}, denoted as $f\equiv g$, 
if $g$ can be obtained from $f$ by permuting the variables and flipping some of the literals.
We denote by $\classe(f)$ the equivalence class of a function~$f$.
\end{definition}
For example, for $n=7$ the function $f$ we have defined before is equivalent to the function
$g= (x_3 \sim x_5 \sim x_2) \wedge (x_1 \sim \bar{x_6})$.
It is easy to check that all the Boolean functions in $\classe(f)$ have the same probability~$\proba(f)$.

\begin{definition}
Let $f \in \xx$; we say that a Boolean variable $x$ is an \emph{essential} variable of~$f$ if and
only if $f|_{x=1} \neq f|_{x=0}$. 
We set $e(f)$ as the number of the essential variables of~$f$.
\end{definition}

\begin{remark}
Although writing the constant functions $\true$ and $\false$ as 2-Xor expressions requires the use
of (at least) one variable, these two functions have no essential variable: $e(\true)=e(\false)=0$.
\end{remark}

Note that $g \not\in \xx_{e(f)-1}$ for all $g$ with $f\equiv g$. But there exists a function $g$
with $f\equiv g$ such that $g \in \xx_{e(f)}$.
In our running example, $e(f) = 5$ 
and the essential variables are $x_1$, $x_2$, $x_3$, $x_5$ and $x_6$,
so we can take, \textit{e.g.}, $g= (x_3 \sim x_5 \sim x_2) \wedge (x_1 \sim \bar{x_6})$.

Again, with the exception of $\false$ that forms a class by itself, the classes we have thus defined on $\xx_n$ are in bijection with the partitions of the integer~$n$; in our example the class of the function $f$ partitions the integer~7 as $1+1+2+3$.

\begin{notation}
Let~$\mathcal{P}(n)$ denote the set of partitions of the integer~$n$.
For any $\ii= (i_\ell)_{\ell \geq 1}$ in~$\mathcal{P}(n)$,
$i_{\ell}$ is the number of parts of size $\ell$.
Hence the size of $\ii$ is 
$s(\ii) := \sum_\ell \ell \, i_\ell=n$,
and the total number of parts (or \emph{blocks}) is 
$\xi(\ii) := \sum_\ell i_\ell$.
A partition of the type $(0, \ldots, 0,1,0,\ldots)$  with the single $1$ in position~$n$ is denoted by $\ii_{\bf max(n)}$.
\end{notation}

We can now express a bijection between
classes of Boolean functions and integer partitions.

\begin{proposition}
Given an integer partition~$\ii$ of~$n$,
let~$\classe_{\ii}$ denote the set of Boolean functions
from~$\xx_n \setminus \{\false\}$
with~$i_{\ell}$ blocks of size~$\ell$ for all~$\ell \geq 1$.
Then~$\{\classe_{\ii}\}_{\ii \in \mathcal{P}(n)}$
is in bijection with the quotient of the set~$\xx_n \setminus \{\false\}$
by the equivalence relation ``$\equiv$''.

We write $\ii(f)$ for the integer partition associated to a Boolean function~$f$, and we extend the notation for the equivalence class into $\classe_\ii = \classe (f)$ when $\ii = \ii(f)$.
\end{proposition}

\begin{proof}
Given a Boolean function~$f$ in~$\xx_n \setminus \{\false\}$,
$\classe(f)$ denotes the class of~$f$ for the equivalence relation~``$\equiv$''.
Therefore, the set of distinct classes~$\classe(f)$
is in bijection with~$(\xx_n \setminus \{\false\})/\equiv$.
Let~$\ii$ denote the integer partition matching the block composition of~$f$.
The demonstration of the proposition is over once we have proven
$\classe_{\ii} = \classe(f)$.

Let us write the block representation of~$f$,
defined in Proposition~\ref{th:block_representation}, as
\begin{align*}
\{
& \{l_{1,1}\}, \{l_{1,2}\}, \ldots, \{l_{1,i_1}\},\\
& \{l_{2,1}, l_{2,2}\}, \{l_{2,3}, l_{2,4}\}, \ldots, \{l_{2,2 i_2-1}, l_{2,2 i_2}\},\\
& \hspace{2.1cm}\vdots\\
& \{l_{t,1}, \ldots, l_{t,t}\}, \ldots, \{l_{t,t i_t - (t-1)}, \ldots, l_{t,t i_t}\}, \ldots \},
\end{align*}
where all~$l_{i,j}$ are literals corresponding to distinct variables.
Let~$g$ be a Boolean function in $\classe_{\ii}$,
with block representation
\begin{align*}
\{
& \{\tilde{l}_{1,1}\}, \{\tilde{l}_{1,2}\}, \ldots, \{\tilde{l}_{1,i_1}\},\\
& \{\tilde{l}_{2,1}, \tilde{l}_{2,2}\}, \{\tilde{l}_{2,3}, \tilde{l}_{2,4}\}, \ldots, \{\tilde{l}_{2,2 i_2-1}, \tilde{l}_{2,2 i_2}\},\\
& \hspace{2.1cm}\vdots\\
& \{\tilde{l}_{t,1}, \ldots, \tilde{l}_{t,t}\}, \ldots, \{\tilde{l}_{t,t i_t - (t-1)}, \ldots, \tilde{l}_{t,t i_t}\}, \ldots \}.
\end{align*}
By flipping some of the literals and permuting the variables,
the block representation of~$f$ can be sent to the block representation of~$g$,
so~$f \equiv g$ and~$\classe_{\ii}$ is a subset of~$\classe(f)$.

Reciprocally, let~$h$ denote a Boolean function in~$\classe(f)$.
By definition, a block representation of~$h$ can be obtained
from the block representation of~$f$ by flipping some literals
and permuting the variables.
Therefore, the block representation of~$h$
corresponds to the same integer partition~$\ii$ as~$f$,
so~$h$ belongs to~$\classe_{\ii}$ and
$\classe(f)$ is a subset of~$\classe_{\ii}$.

Since we have both~$\classe_{\ii} \subset \classe(f)$ 
and~$\classe(f) \subset \classe_{\ii}$, we conclude
that those two sets are equal.
\end{proof}

Our running example corresponds to the integer partition~$(n-5,1,1,0,0,0)$  on $n \geq 5$ variables, which
has $n-3$ parts. The set partition it induces on the set of Boolean variables may be taken, for example, equal to $ \{ x_1,  x_2\}, \{ x_3, x_4, x_5 \}$. The
function $\true$ corresponds to the integer partition~$(n, 0, \ldots, 0)$ 
and is computed by the expressions that have only clauses of the form $l \oplus \bar{l}$.

\begin{proposition}
\label{prop:classes}
\begin{itemize}
\item[i)]
Set $p(n)$ as the number of integer partitions of $n$. Then the number of equivalence classes of computable Boolean functions is $p(n)+1$.

\item[ii)]
The class $C_\ii$ associated to an integer partition $\ii=(i_\ell)$ has cardinality 
\begin{equation} \label{cardinality}
|C_\ii| = \frac{2^{n-\xi(\ii)} \, n!}{\prod_{\ell \geq 1} i_\ell ! (\ell!)^{i_\ell}} .
\end{equation} 

\end{itemize}
\end{proposition}

\begin{remark}
As an aside, we mention that, as $n \to +\infty$ (see~\cite[p.~578]{FlajoletSedgewick}), 
\[
p(n) \sim \frac{1}{4n\sqrt{3}} \exp\(\pi \sqrt{2n/3}\).
\]
\end{remark}

\begin{proof}
The number of classes comes from the bijection between classes, with the exception of the one with
$\false$, and integer partitions, hence we get~i).

To prove ii), note that the number of partitions of the set of the $n$ Boolean variables that 
lead to~$\mathbf{i}$ is 
\[
\frac{n!}{ \prod_{l=1}^n (l!)^{i_l} i_l! }, 
\]
\emph{cf.} \cite[p.~205, Theorem~B]{Co74} or \cite[Theorem 13.2]{An76}. 

Now observe that there are two possible polarities for each variable and hence $2^n$ choices. But
in this way, each block of variables is counted twice, e.g. $x_1 \sim \bar{x_2} \sim x_3$ defines
the same function as $\bar{x_1} \sim x_2 \sim \bar{x_3}$. Hence we have to divide by 2 for each
block and therefore 
the cardinality of the equivalence class $C_\ii$ is given by \eqref{cardinality}.
\end{proof}

\begin{remark}
The factor $2^{n-\xi(\ii)}$ can also be arrived at as follows. Choose a variable in each block
and then fix the polarities of the other variables in this block as equal or opposite to the
chosen variable of the block. This gives $l-1$ decisions for a block of size $l$ and thus in total
a contribution of the multiplicative factor $2^{\sum_{l=2}^{n}i_l(l-1)}$. 
\end{remark}

\subsection{2-Xor Expressions as Colored Multigraphs}\label{EsCM}

In their seminal articles on the first cycle in an evolving graph and the birth of the giant
component, Flajolet, Knuth and Pittel~\cite{FKP89} and Janson, Knuth, \L{}uczak and
Pittel~\cite{giant} introduced the following notions.

The \emph{multigraph process}, also known as the \emph{uniform graph model}, produces a labelled multigraph~$G$ with $n$ vertices and $m$ edges by drawing independently and uniformly $2 m$ vertices in $[1,n]$:
\[
  u_1, v_1, u_2, v_2, \ldots, u_m, v_m.
\]
The set of vertices of $G$ is $V(G) = [1,n]$ and its set of edges is
\[
  E(G) = \{ \{u_1, v_1\}, \{u_2, v_2\}, \ldots, \{u_m, v_m\} \}.
\]
Different drawings can lead to the same multigraph: 
The number of ordered sequences of vertices that correspond to a given multigraph $G$ is
denoted by $\seqv(G)$ and satisfies
\[
  \seqv(G) = |\{ u_1, v_1, \ldots, u_m, v_m \in [1,n]^{2m}\ |\ E(G) = \{ \{u_1, v_1\}, \ldots, \{u_m, v_m\} \} \}|.
\]
A multigraph is \emph{simple} if no edge contains twice the same vertex and all its edges are distinct.
Therefore, it contains neither loops nor multiple edges.
It follows that the number of sequences of vertices that correspond to a given simple multigraph $G$ with $m$ edges is
\[
  \seqv(G) = 2^m m!.
\]
The \emph{compensation factor} $\kappa(G)$ of a multigraph $G$ is classically defined as
\[
  \kappa(G) = \frac{\seqv(G)}{2^m m!},
\]
so a multigraph is simple if and only if its compensation factor is equal to $1$.

\begin{figure}[h]
\begin{center}
       {\includegraphics[width=7cm]{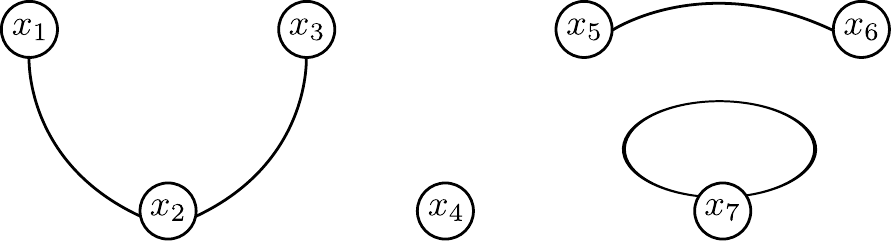}}
\caption{\label{fig:example-multigraph} The multigraph underlying our running example.}
\end{center}
\end{figure}

For example, for $m=4$ and $n=7$ the drawings $x_2$, $x_3$, $x_7$, $x_7$, $x_1$, $x_3$, $x_6$,
$x_5$ and $x_7$, $x_7$, $x_1$, $x_3$, $x_3$, $x_2$, $x_5$, $x_6$ both lead to the multigraph of Figure~\ref{fig:example-multigraph}; indeed the number of ordered sequences leading to this multigraph is $4! \; 2^3 = 192$ and its compensation factor is $\frac{1}{2}$.

\begin{fact}
	Let $\mg_{m,n}$ denote the set of multigraphs  with $n$ vertices and $m$ edges.
  The probability for the multigraph process to produce a multigraph $G$ among all multigraphs in
  $\mg_{m,n}$ is proportional to its compensation factor $\kappa(G)$
  \[
	  \mathds{P}(G\ |\ G \in \mg_{m,n})
    =
    \frac{\kappa(G)}{\sum_{H \in \mg_{m,n}} \kappa(H)}.
  \]
\end{fact}

The \emph{number} of multigraphs in a family $\mathcal{F}$ is defined as the sum of their compensation factors
\[
  \sum_{G \in \mathcal{F}} \kappa(G),
\]
although this quantity might not be an integer.
For example, the total number of multigraphs with $n$ vertices and $m$ edges is
\[
  M_{m,n} = \frac{n^{2m}}{2^m m!},
\]
and the number of cubic multigraphs (\textit{i.e.} multigraphs where all the vertices have degree $3$) with $2 r$ vertices is
\[
  \frac{(6 r)!}{(3!)^{2 r} 2^{3 r} (3 r)!},
\]
because such multigraphs have $3 r$ edges.
If $\mathcal{F}$ contains only simple multigraphs, its number of multigraphs is equal to its cardinality.

Let $n(G)$ and $m(G)$ denote the number of vertices and number of edges of a multigraph $G$,
respectively. 
The generating function corresponding to a family $\mathcal{F}$ of multigraphs is 
\[
  \sum_{G \in \mathcal{F}}
  \kappa(G) z^{m(G)} \frac{v^{n(G)}}{n(G)!}.
\]
For example, the generating function of all multigraphs is
\[
  M(z,v) = \sum_{n \geq 0} e^{\frac{n^2}{2} z} \frac{v^n}{n!}.
\]
As already observed by Janson, Knuth, \L{}uczak and Pittel\cite{giant}, 
and Flajolet, Salvy and Schaeffer \cite{FSS04},
a multigraph is a set of connected multigraphs, 
so the generating function for connected multigraphs is
\[
  C(z,v) = \log M(z,v) = \sum_{r \geq -1} z^r \; C_r (z v)
\]
where we have set $r = m-n$, the \emph{excess} of the multigraph, and where $C_r(z)$ is the
generating function associated with \emph{connected} multigraphs of fixed excess $r$.

\medskip
We are now ready to define a \emph{bijection} between Boolean expressions  and \emph{colored} 
multigraphs, i.e. multigraphs with different types (colors) of edges between any two vertices.

\begin{figure}[h]
\begin{center}
       {\includegraphics[width=7cm]{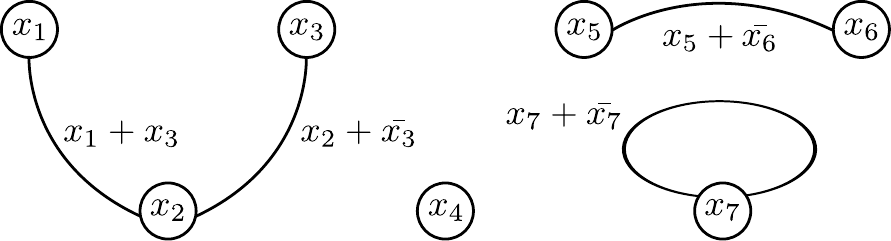}}
\caption{\label{fig:example} The colored multigraph for our running example.}
\end{center}
\end{figure}

\begin{proposition}
\label{fg-bijection}
The 2-Xor expressions are in bijection with multigraphs where loops are 4-colored and other edges are 8-colored.
This bijection is such that, for all $f\in \xx$  the number of connected components of the associated  multigraph is  $\xi(\ii(f))$.
Thus the function $M(8 z,v)$ is the bi-exponential generating function for 2-Xor expressions, i.e. 
\[
	M(8 z,v)=\sum_{n\ge 0}\sum_{m\ge 0} \N \frac{z^n v^m}{n! m!}.
\]
\end{proposition}

\begin{proof}
We first present the bijection between a 2-Xor expression of $m$ clauses on $n$ variables, and a  colored multigraph on $n$ vertices and with $m$ edges.

\begin{itemize}
\item
Each Boolean variable $x_\ell$ corresponds to a vertex, and each 2-Xor clause to an edge between two distinct vertices, or to a loop on one vertex; each loop or edge can be repeated.

\item
A loop on vertex $x$ has one of four colors: $x \oplus x$, $x \oplus \bar{x}$, $\bar{x} \oplus x$ or $\bar{x} \oplus \bar{x}$.

\item
An edge between two distinct vertices $x_i$ and $x_j$ has one of eight colors: $l_i \oplus l_j$ or $l_j \oplus l_i$, where $l_i$ and $l_j$ are respectively equal to $x_i$ or its negation, and $x_j$ or its negation.

\end{itemize}
It is then an easy matter to check that the number of connected components of the multigraph is simply the number of parts in the integer partition associated with the function~$f$ computed by the expression.

We next turn to the generating function for 2-Xor expressions and start from the generating function for multigraphs
\[
M(z,v) = \sum_{m,n} \M \frac{v^n}{n!} \, z^m
= \sum_{n \geq 0} e^{\frac{n^2}{2} \, z} \; \frac{v^n}{n!},
\]
with $v$ marking the vertices and $z$ marking the edges and loops, and $\M$  the number of multigraphs on $n$ vertices and with $m$ edges.
Consider expressions built on $n$ variables, and set $\N$ as the number of such expressions with $m$ clauses.
Each vertex contributes a term $e^{4z}$ for the loops: There are 4 possible colors;
each vertex $x$ also contributes a term $\prod_{y: x<y\leq n} e^{8z} = e^{8z (n-x)}$ for the edges
to a different vertex: There are 8 possible colors. We order the vertices so as not to count them
twice. Taking into account all $n$ vertices gives
$$
\sum_m \N\, \frac{z^m}{m!}  = \prod_{s=1}^n e^{4z} \; e^{8z (n-s)}
= e^{4 n^2 z} ,
$$
which in turn leads to an expression for the global generating function as
\[
 \sum_{m,n} \N\, \frac{z^m}{m!} \; \frac{v^n}{n!}
= \sum_n e^{4 n^2 \, z} \; \frac{v^n}{n!} = M(8z,v).\qedhere
\]
\end{proof}

\subsection{The Different Ranges}
\label{sec:ranges}
We shall not consider the whole range of values for the parameters $n$ and $m$ when studying the probabilities on $\xx_n$, but restrict our investigations to the case where $m$ and $n$ are (roughly) proportional -- which is the most interesting part as it includes the domain around the threshold -- and set $m\sim\alpha n$ ($\alpha$ is usually assumed to be a constant). 
It is well known (see, e.g., \cite{DaudeRavelomanana}) that the probability that a random expression is satisfiable decreases from 1 to 0 when $\alpha$ increases, with a (coarse) threshold at $\frac{1}{2}$. 
However \emph{a Boolean function corresponding to a partition of the $n$ Boolean variables into $p$ blocks cannot appear before at least $n-p$ clauses have been drawn, i.e. before $m \geq n-p$}.
E.g., the function $x_1 \sim \cdots \sim x_n$ cannot appear for $m < n-1$, which means that it has a non-zero probability only for $\alpha \geq 1$, much later than the threshold -- and at this point the probability of $\false$ is~$1-o(1)$.
This leads us to define several regions according to the value of the ratio~$\alpha=m/n$ when $m,n \rightarrow +\infty$:
\begin{itemize}
\item $\alpha < 1/2$.
Here the probability of satisfiability is non-zero, but the attainable functions cannot have more than $n(1-\alpha)$ blocks.

\item $\alpha = 1/2$.
This is precisely the threshold range.

\item $1/2 < \alpha < 1$.
Some Boolean functions still have probability zero, but now the probability of satisfiability is $o(1)$ and the probability of $\false$ is $1-o(1)$.
Thus any other attainable Boolean function has a vanishing probability~$o(1)$.

\item $1 \leq \alpha$.
At this point all the attainable Boolean functions have non-zero probability, but again the
probability of $\false$ is tending to 1.

\end{itemize}

\section{Probabilities on the Set of Boolean Functions}
\label{sec:probas}

We consider here how we can obtain the probability of satisfiability (or equivalently of $\false$), or of any function in $\xx_n$. 
The reader should recall that the probabilities given in the sequel are actually distributions on $\xx_n$, i.e. they depend on~$n$ and~$m$.
Letting $n$ and $m=m(n)$ grow to infinity amounts to specializing the probability distribution
$\proba(f)$ (defined in Section~\ref{sec:model-expressions} for $f \in \xx_n$) to ${\mathrm
Pr}_{[m(n),n]}(f)$.
We shall be interested in its limit when $n \rightarrow + \infty$ and $f$ is a function of~$\xx$.
First we will consider the case $f=\false$ (which is the usual satisfiability problem) and derive anew the probability of satisfiability in the critical window, before turning to general Boolean functions.
We begin with some enumeration results on multigraphs that will be useful in the proofs of our results.

\begin{remark}
Note that the classical satisfiability problems as well as the above described extension are
looking for the limit of the probability ${\mathrm Pr}_{[m(n),n]}(f)$, as $n\to\infty$. This
raises the question whether the squence of distributions ${\mathrm Pr}_{[m(n),n]}$ defines a
limiting distribution on the set $\mathcal X$. We do not know whether this is true or not, but our
asymptotic results either concern the limit of the probabilities ${\mathrm
Pr}_{[m(n),n]}(f)$ for some \emph{a priori} given function $f$ which is independent of $n$ (lying
in some $\mathcal X_{n_0}$; then the limit for $n\to\infty$ is taken) or a particular sequence of
function which depends on $n$. 

When looking into the literature of quantitative logic, the question for certain limiting
probabilities often arises and is settled by means of the Drmota-Lalley-Woods theorem (see
\cite[p.~489]{FlajoletSedgewick} for the polynomial version and \cite[Sec.~2.2.5]{Drmota09} for
the analytic version). In order to apply this theorem, one has to specify the problem in terms of
a system of functional equations which has certain technical properties, in particular it must not
be linear. 
Usually, for each Boolean function one defines a generating associated with the expressions
representing the Boolean function and sets up a sort of a recursive description of the Boolean
function in terms of the other Boolean functions. If we do that for $2$-Xor formulas, we get a
linear system of functional equations, which is therefore not covered by the Drmota-Lalley-Woods
framework. Despite linearity, the system is complicated to analyze, and so we decided to approach
the problem through a bijection to certain classes of multigraphs and exploit the rich existing
knowledge on multigraph generating functions. 
\end{remark}

    \subsection{Asymptotics for Multigraphs} \label{sec:asymptotics-multigraphs}


    \subsubsection{Connected Multigraphs}

Connected graphs with a large number of vertices 
have been counted for various ranges of number of edges.
The first result is attributed to Cayley,
who obtained in~$1889$ an exact formula for the number 
of unrooted trees by resolution of a recurrence
(see~\cite[p.~51]{BLW74} for a historical discussion by Biggs, Lloyd and Wilson).
R\'enyi~\cite{ER59} derived an asymptotic formula 
for the number of unicyclic graphs. 
Erd\H{o}s and R\'enyi obtained in~\cite{ER59} 
the probability for a random graph
with high density of edges to be connected.
From their result follows an expression for the asymptotic
number of connected graphs with $n$ vertices
and $m$ edges when~$m-n = \frac{1}{2} n (\log(n) + c)$
for any value~$c$ fixed or growing to infinity.
Using generating functions, 
Wright~\cite{Wri1977} gave the asymptotic number of connected graphs
for~$m-n$ arbitrary but fixed,
and also studied the case~$m-n = \smallo(n^{1/3})$ in~\cite{Wri1980}.
Finally, Bender, Canfield, and McKay~\cite{BCMK90}
obtained the asymptotic number of connected graphs for all $n, m-n \rightarrow \infty$.
Their proof is based on a recursive formula derived by Wright. 
New proofs were proposed in~\cite{PW05} and~\cite{HS06}.

For historical reasons, most of those results
were only stated for simple graphs.
In the following theorems, we summarize those results 
and adapt them to multigraphs.

\begin{notation} \label{th:CmnCr}
The number of connected multigraphs
with $n$ vertices and $m$ edges is denoted by $C_{m,n}$.

The exponential generating function
of connected multigraphs with excess $r= m-n$
is denoted by
\[
  C_r(v) = \sum_{n \geq 0} C_{n+r,n} \frac{v^n}{n!}.
\]
\end{notation}

\begin{theorem}
\label{th:connected-multigraphs-fixed-excess}
  When the excess $r=m-n$ is fixed, then 
  \begin{equation} \label{eq:asympt-bernhard-exces}
    C_{m,n} \sim K_r n^{n + \frac{3 r - 1}{2}},
  \end{equation} 
  where the value of $K_r$ is
  \[
    K_r = 
    \begin{cases}
      1  & \text{if } r = -1, \\
      \frac{\sqrt{2\pi}}{4}  & \text{if } r = 0,\\
      \frac{ \sqrt{2 \pi} }{ 2^{3r/2} \Gamma( 3r/2 ) } 
      [v^{2r}] \log \left(\sum_{\ell \geq 0} \frac{(6\ell)!}{288^\ell (3\ell)!} \frac{v^{2\ell}}{(2\ell)!} \right) & \text{if } r > 0.
    \end{cases}
  \]    
\end{theorem}

\begin{remark}
Note that the excess of a connected multigraph is always greater or equal to $-1$.
\end{remark}

\begin{proof}
\begin{itemize}
\item
  For $r=-1$, the connected component is an unrooted tree, 
  $C_{-1}(v) = T(v) - T(v)^2 /2$ where $T(v)= \sum_n n^{n-1} \frac{v^n}{n!}$ 
  is the so-called tree function, and  \cite[p.~132]{FlajoletSedgewick}:
  \[
    n! [ v^n] C_{-1}(v) = n^{n-2}.
  \]
\item
  For $r=0$, the connected component is unicyclic, $C_{0}(v) = \frac{1}{2} \log \frac{1}{1-T(v)}$ and 
  (again from \cite[p.~133]{FlajoletSedgewick}):
  \[
    n! [ v^n] C_{0}(v) \sim \frac{1}{4} n^{n-1} \sqrt{2 n\pi} .
  \]
\item
  For $r \geq 1$, we follow the approach of Wright~\cite{Wri1977}.
  A \emph{kernel} is a multigraph with minimum degree at least $3$.
  Let us define the \emph{deficiency} of a kernel of excess $r$ with $n$ vertices
  as $d = 2 r - n$. it follows that the number of edges of a kernel is $m = 3r - d$.
  Let also $C^{(\geq 3)}_{r,d}$ denote the number of connected kernels of excess $r$ and deficiency $d$.
  All connected multigraphs of excess $r \geq 1$
  can be build from the connected kernels of excess $r$
  by replacing the edges with paths
  and the vertices with rooted trees, so
  \[
    C_{r}(v) = \sum_{d=0}^{2r-1} \frac{C^{(\geq 3)}_{r,d}}{(2 r - d)!} . \frac{T(v)^{2r-d}}{(1-T(v))^{3r-d}},
  \]
  which gives
  \begin{equation}
    \label{eq:asymptotic-phi-n}
    C_{n+r,n} = n! [v^n] C_{r}(v) = \sum_{d=0}^{2r-1} \frac{C^{(\geq 3)}_{r,d}}{(2 r - d)!} . [ v^n ] \frac{T(v)^{2r-d}}{(1-T(v))^{3r-d}}.
  \end{equation}
  We must compute the coefficients $[ v^n ] \frac{T(v)^{2r-d}}{(1-T(v))^{3r-d}}$.
  We have, for any fixed positive integers $a$ and $b$,
  \[
    n! [v^n] \frac{T(v)^a}{(1-T(v))^b} \sim 
    \frac{2^{-b/2}}{\Gamma(b/2)} \, e^n \, n^{b/2-1} \, n!,
  \]
  which is independent of~$a$.
  When $r$ is fixed, we see that, of the $2r$ terms in Equation~(\ref{eq:asymptotic-phi-n}), 
  the one for $d=0$ gives the dominant term and we get, also from \cite[p.~134]{FlajoletSedgewick}:
  \[
    n! [v^n] C_{r}(v) \sim 
    \frac{C^{(\geq 3)}_{r,0}}{(2 r)!} \, \frac{\sqrt{2\pi}}{2^{3r/2} \, \Gamma(3r/2)} \; n^{n+\frac{3r-1}{2}}.
  \]
  Finally, the constant $C^{(\geq 3)}_{r,0}$ is 
  the number of connected cubic multigraphs 
  (\textit{i.e.} $3$-regular multigraphs).
  Since there are $\frac{(6\ell)!}{(3!)^{2 \ell} 2^{3 \ell} (3 \ell)!}$ cubic multigraphs with $2
  \ell$ vertices (see Section~\ref{EsCM}),
  the generating function of connected cubic multigraphs is
  \[
    \sum_{\ell \geq 1} C^{(\geq 3)}_{\ell,0} \frac{v^{2\ell}}{(2\ell)!} =
    \log \left( \sum_{\ell \geq 0} \frac{(6\ell)!}{288^\ell (3\ell)!} \frac{v^{2\ell}}{(2\ell)!} \right),
  \]
  and a coefficient extraction leads to
  \[
    \frac{C^{(\geq 3)}_{r,0}}{(2 r)!} =
    [v^{2 r}] \log \left( \sum_{\ell \geq 0} \frac{(6\ell)!}{288^\ell (3\ell)!} \frac{v^{2\ell}}{(2\ell)!} \right).
\qedhere
  \]
\end{itemize}
\end{proof}

When the excess~$r$ goes to infinity,
non-cubic kernels cease to be negligible,
and a different approach is needed
to enumerate the connected multigraphs.

\begin{theorem}
\label{th:multigraphs-large-excess}
  When $m-n$ goes to infinity and $\frac{2m}{n} - \log(n)$
  tends towards a constant or~$- \infty$, 
  the asymptotic number of connected multigraphs is
  \[
    C_{m,n} =
    \sqrt{\frac{2(e^\lambda - 1-\lambda)^2}{\lambda (e^{2\lambda} -1 - 2\lambda e^\lambda)}} 
    \frac{n^m}{\sqrt{2 \pi n}}
    \frac{\left( 2 \sinh(\lambda /2) \right)^n}{\lambda^m}
    \left( 1 + \bigO \left( (m-n) e^{-2 m/n} \right)^{-1/2+\epsilon} \right)
  \]
  for any $\epsilon > 0$,
  where the value $\lambda$ is characterized by the relation
  \[
    \frac{\lambda}{2} \coth \frac{\lambda}{2} = \frac{m}{n}.
  \]
\end{theorem}
\begin{proof}
This asymptotic expression has already been derived for simple graphs.
Unfortunately, the corresponding proofs are too long
to be reproduced and adapted here for multigraphs.
Instead, we follow the proof from Pittel and Wormald~\cite{PW05}.
and indicate the necessary changes in order to obtain
the same result for multigraphs.
A proof based on analytic combinatorics
is also available in~\cite[Theorem~$5.1.3$]{ElieThesis}.
It is however restricted to the case where~$2m/n$ tends toward a constant.

The proof starts with the enumeration of \emph{cores},
which are multigraphs with minimum degree at least~$2$.
Cores correspond to sequences of vertices
\[
  u_1, v_1, \ldots, u_m, v_m
\]
where each vertex appears at least~$2$ times.
The number of such sequences of length $2m$
with $n$ vertices is
\[
  \sum_{\substack{d_1, \ldots, d_n \geq 2 \\ d_1 + \cdots + d_n = 2 m}}
  \binom{2 m}{d_1, \ldots, d_n}
  =
  (2m)! Q(n,m),
\]
where the quantity $Q(n,m)$ is defined in~\cite[Equation~(2.1)]{PW05} by
\[
  Q(n,m) = 
  \sum_{\substack{d_1, \ldots, d_n \geq 2 \\ d_1 + \cdots + d_n = 2 m}}
  \prod_{j=1}^n \frac{1}{d_j!}. 
\]
Therefore, the number
of cores with $n$ vertices and $m$ edges,
defined as the sum of their compensation factors, is
\[
  \core_{m,n} = 
  \frac{(2m)!}{2^m m!}
  Q(n,m),
\]
which replaces Equation~(3.9) of~\cite[Theorem~$8$, p.~13]{PW05}.
Its asymptotic estimate, given in \cite[Equation~(3.11), p.~13]{PW05} is now
\[
  \core_{m,n} =
  (1 + \bigO((m-n)^{-1} + (m-n)^{1/2} n^{-1 + \epsilon})) 
  \frac{(2m-1)!! f(\lambda)^n}{\lambda^{2m}}
  \frac{1}{\sqrt{2 \pi n c (1+\bar{\eta} - c)}}
\]
where $\lambda$, $f$, $c$ and $\bar{\eta}$ have the same definition as in~\cite{PW05}.

The second step of the proof is the enumeration of cores
that contain no isolated cycles.
Let $\corenocycle_{m,n}$ denote the number of such multigraphs
with $n$ vertices and $m$ edges.
The result is stated in \cite[p.~4, Theorem~2]{PW05}
and its proof can be found in \cite[Section~6]{PW05}.
It relies on the exponential generating function of simple undirected cycles
\[
  \sum_{\ell \geq 3} \frac{x^\ell}{2 \ell}
  = - \frac{1}{2} \log(1-x) - \frac{x}{2} - \frac{x^2}{4}.
\]
In multigraphs, a cycle might also have size $1$ (a loop),
or size $2$ (a double edge), 
so we replace the previous generating function with
\[
  \sum_{\ell \geq 1} \frac{x^\ell}{2 \ell}
  = - \frac{1}{2} \log(1-x)
\]
and replace the function $h(x)$, defined in \cite[p.~4, Equation~(2.3)]{PW05}, by
\[
  h(x) 
  = e^{- \sum_{\ell \geq 1} \frac{x^\ell}{2 \ell}}
  = (1-x)^{1/2}.
\]
Theorem~$2$ of~\cite[p.~4]{PW05} becomes for multigraphs:
``\emph{when $m-n$ goes to infinity and $m = \bigO(n \log(n))$,
then for any fixed $\epsilon > 0$,
the number of cores with $n$ vertices and $m$ edges
that contain no isolated cycles is
\[
  \corenocycle_{m,n} = 
  (1 + \bigO(n^{-1/2 + \epsilon} + (m-n)^{-1})) 
  h \left( \frac{\lambda}{e^\lambda - 1} \right) 
  \core_{m,n} 
\]
where $\lambda$ is the unique positive root of
$
  \frac{\lambda(e^\lambda-1)}{e^\lambda-1-\lambda} = \frac{m}{n}
$.
}''

The last ingredient of the proof is an observation from Erd\H{o}s and R\'enyi,
that when $m-n$ tends to infinity, almost all graphs or multigraphs
that contain neither trees nor unicyclic components are connected.
Therefore, $C_{m,n}$ is asymptotically equal to the number of such multigraphs.
They correspond to the cores without isolated cycle,
where each vertex is replaced with a rooted tree.
Their exact number is derived in \cite[Equation~(7.1)]{PW05}
and becomes for multigraphs
\[
  \sum_{\mu=1}^n 
  \binom{n}{\mu} 
  \mu 
  n^{n - \mu - 1}
  \corenocycle_{m-n+\mu,\mu}.
\]
Borrowing the notation of~\cite{PW05},
the summand is estimated in \cite[Equation~(7.2)]{PW05} by
combining \cite[Theorem~2]{PW05} and \cite[Equation~(3.11)]{PW05},
which we both have modified
\[
  \binom{n}{\mu} 
  \mu 
  n^{n - \mu - 1}
  \corenocycle_{m-n+\mu,\mu}
  =
  (1+\bigO(\beta_1)) 
  n^m
  F_n(y)
  \exp(n H(y, \lambda)).
\]
The adaptation for multigraphs only recquires to change the definition of $F_n(y)$
and replace it with
\[
  F_n(y) = 
  \frac{1}{2 \pi n}
  \sqrt{
    \frac{(1-\sigma) y}{
      u(1+\bar{\eta} - 2 u/y)(1-y+\rho)
    }
  },
\]
using again the notations $u$, $c$, $\lambda$, $\bar{\eta}$, $\rho$ and $\sigma$ of~\cite{PW05}.
The rest of the proof is a Laplace method.
The modification we made in the definition of $F_n$ also impacts 
\cite[p.~31, Equation~(7.16)]{PW05} which becomes
\[
  F_n(\bar{y}) = 
  \frac{ 
    \sqrt{2} ( e^{\bar{\lambda}} -1 - \bar{\lambda} )^{3/2}
  }{
    2 \pi n \bar{\lambda}
    \sqrt{ (e^{\bar{\lambda}}-1)^2 - \bar{\lambda}^2 e^{\bar{\lambda}} }
  }.
\]
As a consequence, the definition of the value~$\alpha$ of~\cite[p.~5, Theorem~3]{PW05}
is, for multigraphs,
\[
  \alpha = 
  \sqrt{ 
  \frac{ 2 ( e^{\bar{\lambda}} - 1 - \bar{\lambda} )^2 }{ 
    \bar{\lambda} ( e^{2 \bar{\lambda}} - 1 - 2\bar{\lambda}e^{\bar{\lambda}} ) 
  } 
  }
\]
while the other quantities of the theorem stay unchanged.
\end{proof}

Remark that the value $\lambda$ of the previous theorem
is a constant only when $\frac{m}{n}$ is fixed.

As observed by Pittel and Wormald in~\cite{PW05},
the asymptotic formula of the previous theorem also holds
when $\frac{2m}{n} - \log(n)$ tends slowly towards infinity.
However, we do not need this extension, because
this range of $m$ is already covered by the following theorem.

\begin{theorem} \label{connected_excess-infinite}
  When both $m-n$ and $\frac{2 m}{n} - \log(n)$ go to infinity,
  the asymptotic number of connected multigraphs becomes
  \[
    C_{m,n} \sim
    \frac{n^{2m}}{2^m m!}.
  \]
\end{theorem}
\begin{proof}
  Erd\H{o}s and R\'enyi proved in~\cite{ER59}
  that when~$2m/n - \log(n)$ goes to infinity, 
  a random multigraph with $n$ vertices and $m$ edges
  is connected with high probability.
  Therefore, the number of connected multigraphs 
  is then asymptotically equal to the total number of multigraphs, $\frac{n^{2m}}{2^m m!}$.
\end{proof}

    \subsubsection{Weighted Multigraphs}

As recalled in the definition of the multigraph process,
multigraphs are counted according to their compensation factor,
meaning that the number of multigraphs in a family~$\mathcal{F}$
is defined as the sum of their compensation factors~$\sum_{G \in \mathcal{F}} \kappa(G)$.
The proof of Theorems~\ref{th:proba(sat)} and~\ref{th:proba-random-input} 
require a refinement of this definition,
involving the number of connected components of the multigraphs.
Specifically, we now count the number of multigraphs with $n$ vertices
and $m$ edges according to their compensation factor and a factor $\sigma$
for each connected component
\[
  \sum_{G \in \mg_{m,n}} \kappa(G) \sigma^{c(G)}
\]
where $\sigma$ is a positive real value 
and $c(G)$ denotes the number of components of~$G$.
Since the generating function of connected multigraphs is $\log M(z,v)$
and a multigraph is a set of connected multigraphs,
the previous quantity can be expressed by a coefficient extraction
\[
  \sum_{G \in \mg_{m,n}} \kappa(G) \sigma^{c(G)}
  =
  n! [z^m v^n]
  e^{\sigma \log M(z,v)}
  =
  n! [z^m v^n]
  M^\sigma(z,v).
\]
We list asymptotic formulas for those values
in the following lemma, which combines
Theorems~$8$, $9$ and~$10$ of~\cite{PR14}.
The first part focuses on multigraphs
with less edges than half the number of vertices.
As proved by Erd\H{o}s and R\'enyi, 
with high probability, they contain
only trees and unicyclic components.
The second part investigates the critical window
where the number of edges is around half the number of vertices.
In this range, connected components with fixed excess appear.
Higher number of edges seem more technical to analyze.
However, the probability of satisfiability
of the corresponding $2$-Xor formulas has already reached~$0$ almost surely,
and its study is therefore less interesting.

\begin{lemma} \label{th:sigmacritical}
  Let $\sigma$ denote a fixed positive value.
  When $\frac{m}{n}$ is in a fixed closed interval of~$]0, 1/2[$, then
  \[
    n! [z^m v^n] M^{\sigma}(z,v) 
    \sim 
    \frac{n^{2m}}{2^m m!}  
    \sigma^{n-m}
    \left( 1 - \frac{2m}{n} \right)^{\frac{1-\sigma}{2}}.
  \]
  When $m=\frac{n}{2} (1 + \mu n^{-1/3})$ 
  and~$\mu$ is bounded, then
  \[
    n! [z^m v^n] M^{\sigma}(z,v) 
    \sim 
    \frac{n^{2m}}{2^m m!}
    \sigma^{n-m}
    n^{\frac{\sigma-1}{6}}
    \sum_{r \geq 0}
    \sigma^{r}
    e^{(\sigma)}_r
    \sqrt{2 \pi}
    A(3 r + \sigma/2, \mu),
  \]
  where the value of $e^{(\sigma)}_r$ is
  \[
    e^{(\sigma)}_r = 
    [z^{2r}] 
    \left(
      \sum_{k \geq 0} 
      \frac{(6k!)}{2^{5k} 3^{2k} (3k)!}
      \frac{z^{2k}}{(2k)!}
    \right)^{\sigma}
  \]
  and the function $A$ is defined in~\cite[Lemma~3]{giant} by
  \[
    A(y, \mu) = 
    \frac{e^{-\mu^3/6}}{3^{(y+1)/3}}
    \sum_{k \geq 0} 
    \frac{(3^{2/3} \mu / 2)^k}{k! \Gamma\(\frac{y+1-2k}3\)}.
  \]
\end{lemma}

\begin{remark}
In the rest of the paper, we will only need the cases~$\sigma = 1$ and $\sigma = 1/2$.
\end{remark}
 
The function~$A(y,\mu)$ is a variation 
of the classical Airy function which has been 
thoroughly analyzed in~\cite[Lemma~3]{giant}.
For example, as mentioned in~\cite[Equation~(10.28)]{giant},
for $y=1$ it satisfies the relation
\[
  A(1,\mu) = e^{-\mu/12} \operatorname{Ai}(\mu^2/4),
\]
and for $y=0$, it holds that
\[
  A(0,\mu) = 
  - \frac{1}{2} \mu e^{- \mu^3/12} \operatorname{Ai}(\mu^2/4)
  - e^{-\mu^3/12}
  \operatorname{Ai}'(\mu^2/4).
\]
It is also close to the function defined in~\cite[Theorem~11]{BFSS01} 
and~\cite[Theorem~IX.16]{FlajoletSedgewick}, 
denoted by $G$ in the first one, and by~$S$ in the second one.

    \subsection{Probability of Satisfiability}

The probability of satisfiability
of a random $2$-Xor expression
has been studied by Creignou and Daud\'e~\cite{CreignouDaude99, CrDa04},
Daud\'e and Ravelomanana~\cite{DaudeRavelomanana}
and Pittel and Yeum~\cite{PittelYeum}.
We derive anew their results
to give a first application
of the link between $2$-Xor expressions
and colored multigraphs.

\begin{theorem}\label{th:proba(sat)}
The probability that a random expression is satisfiable is 
\[
  \proba({\mathrm Sat}) = \frac{[z^m v^n] \sqrt{M(4z,2v)}}{[ z^m v^n] M(8z,v)}.
\]
Its limit for $n \rightarrow + \infty$
when $\frac{m}{n}$ is in a fixed closed interval
of $]0, \frac{1}{2}[$ is
\[
  \left( 1 - \frac{2m}{n} \right)^{1/4}.
\]
When $m = \frac{n}{2}(1 + \mu n^{-1/3})$
and $\mu$ is bounded,
this becomes
\[
  n^{-1/12} \sqrt{2 \pi}
  \sum_{r \geq 0} \frac{e_r^{(1/2)}}{2^r} A(3r + 1/4, \mu),
\]
with the notations of Lemma~\ref{th:sigmacritical}.
\end{theorem}
\begin{proof}
To obtain the generating function for satisfiable expressions, 
we shall count the number of pairs 
$\{$satisfiable expression, satisfying as\-sign\-ment$\}$, 
then get rid of the number of satisfying assignments. 
We can assign $\true$ or $\false$ to each variable, 
and one of eight colors to an edge,
hence $M(8z,2v)$ is the generating function associated with the 
pairs $\{$expression, as\-sign\-ment$\}$.

Once we have chosen an assignment of variables, 
for an expression to be satisfiable 
we have to restrict the edges we allow.
Say that $x$ and $y$ are assigned the same value; 
then the edges colored by 
$x\oplus y$, $y \oplus x$, $\bar{x} \oplus \bar{y}$ or $ \bar{y} \oplus \bar{x }$ 
cannot appear in a satisfiable expression.
For a similar reason, the only loops allowed 
are $x \oplus \bar{x}$ or $\bar{x} \oplus {x}$.
We thus count multigraphs with $2$ colors of loops 
and $4$ colors of edges, which gives a generating function equal to $M(4z,2v)$.
  
Now consider the generating function $S(z,v)$ 
for satisfiable expressions: 
We claim that it is equal to $\sqrt{M(4z,2v)}$.
To see this, choose an expression 
computing a Boolean function~$f$, 
and consider how many assignments satisfy it: 
We have seen (cf. Proposition~\ref{prop:classes}) 
that their number is equal to $2^ {\xi(f)}$, 
with $\xi(f)$ the number of connected components 
(once we have chosen the value of a single variable in a block, 
all other variables in that block have received their values 
if the expression is to be satisfiable).
This means that, writing 
$S(z,v) = \exp \left(\log S(z,v)\right)$ with $\log S(z,v)$ 
the function for connected components,
the generating function enumerating the pairs 
$\{$expression, satisfiable as\-sign\-ment$\}$ 
is equal to $\exp (2 \log S(z,v)) = S(z,v)^2$.
As we have just shown that it is also equal 
to $M(4z,2v)$, the value of $\proba(Sat)$ follows.

To obtain the asymptotics before and in the critical window~$m = n/2 + \bigO(n^{2/3})$,
we use Lemma~\ref{th:sigmacritical}. 
\end{proof}

The link between the enumeration
of $2$-Xor expressions and of multigraphs
and the knowledge of the asymptotic number of multigraphs
can also be combined to investigate
the probability for a satisfiable expression
to be satisfied by an input.

\begin{theorem}
\label{th:proba-random-input}
The probability that an input (fixed or random) 
satisfies a random satisfiable expression
with $n$ variables, $m$ clauses and excess $r = m-n$ is 
\[
  \frac{[z^m v^n] M(4z,2v)}{2^n [z^m v^n] \sqrt{M(4z,2v)}}.
\]
When $\frac{m}{n}$ is in a closed interval of $]0,\frac{1}{2}[$, then this is asymptotically equivalent to 
\[
  \frac{1}{2^m}
  \left( 1 - \frac{2m}{n} \right)^{-1/4},
\]
and it is
\[
 \frac{n^{1/12}}{2^m}
 \frac{1}{\sum_r 2^{-r} e_r^{(2)} A(3r + 1/4, \mu)}.
\]
for $m = \frac{n}{2}(1 + \mu n^{-1/3})$ with $\mu$ bounded,
using the notation of Lemma~\ref{th:sigmacritical}.
\end{theorem}

\begin{proof}
The probability that a random expression is satisfied 
by a random assignment is equal to the number of pairs 
$\{$satisfiable expression, satisfying assignment$\}$, 
divided by the number of satisfiable expressions 
and by the number $2^n$ of assignments.
The exact value follows from the fact that 
the generating functions for the number of satisfiable expressions 
and for the number of pairs $\{$satisfiable expression, satisfying assignment$\}$
are respectively $\sqrt{M(4z,2v)}$ and $M(4z,2v)$;
the asymptotic approximations come again 
from Lemma~\ref{th:sigmacritical}.
\end{proof}

    \subsection{Probability of a Given $2$-Xor Function}
    \label{sec:res-general}

We now refine the probability of satisfiability, by computing the probability of a specific Boolean function $\neq \false$.
We first give in Proposition~\ref{prop:fgs} the generating functions for all Boolean functions
(except again $\false$), then use it to provide a general expression for the probability of a
Boolean function in Theorem~\ref{th:proba-f}, or rather of all the functions of an equivalence class~$C_\ii$.
This theorem is at a level of generality that does not give readily precise probabilities, and we delay until Section~\ref{sec:results} such examples of asymptotic probabilities.

\begin{proposition}
\label{prop:fgs}
  Let $f$ denote a Boolean function in~$\xx$
  and $\ii(f)$ the corresponding integer partition.
  Define $\phi_{\ii(f)} (z)$ as the generating function 
  for Boolean expressions that compute~$f$:
  \[
    \phi_{\ii(f)} (z)= \sum_m \N(f) \frac{z^m}{m!}.
  \]
  When $\ii = \ii_{\bf max} (n)$, we set $\phi_n(z) := \phi_{\ii_{\bf max}(n)}(z)$.
  Then
  \begin{eqnarray*}
    \phi_n (z)= n! [v^n] C(4z,v) ;
    \qquad
    \phi_{\ii(f)} (z) = \prod_{\ell \geq 1} \left( \ell! [v^{\ell}] C(4z,v) \right)^{i(f)_\ell}.
  \end{eqnarray*} 
\end{proposition}

\begin{proof}
A canonical representant of the class $\ii_{\bf max} (n)$ is the function $x_1 \sim \cdots \sim x_n$.
Any expression that computes it corresponds to a connected multigraph, where we only allow the 2 types of loops that compute $\true$ and the 4 types of edges between $x_i$ and $x_j$ ($i \neq j$) that compute $x_i \sim x_j$; this gives readily the expression of $\phi_n(z)$.

As for functions whose associated multigraphs have several components, such multigraphs are a product of connected components; hence the global generating function is itself the product of the generating functions for each component.
\end{proof}

\begin{theorem}
\label{th:proba-f}
\begin{enumerate}
\item
  The probability that a random expression of $m$ clauses on $n$ variables 
  computes the function $x_1 \sim \cdots \sim x_n$~is
  \[
  \proba (x_1 \sim \cdots \sim x_n) = 
  \frac{ m! n! [z^m v^n] C(4z,v)}{m! n! [z^m v^n] M(8z,v)}
  =
  \frac{m!}{n^{2m}} \; n! [v^n] C_{m-n}(v).
  \]

\item
  Let $f$ be a function of $\xx$, with $q= \sum_\ell \ii(f)_\ell$, 
  and $B_1, \ldots, B_q$ be the blocks of $\ii(f)$, 
  with $r_j$ ($1\leq j \leq q$) the excess of the block~$B_j$.
  The probability that a random expression of $m$ clauses on $n$ variables computes~$f$~is
  \begin{eqnarray*}
    \proba (f) &=& 
    \frac{m!}{n^{2m}} \;
    \sum_{ \substack{r_1, \ldots ,r_q \geq -1 \\ r_1+ \cdots +r_q=m-n} } \;
    \prod_{j=1}^q
    |B_j|! [ v^{|B_j|}]  C_{r_j}(v),
  \end{eqnarray*}
  where~$C_r(v)$ denote the generating function
  of connected multigraphs of excess~$r$,
  defined in Notation~\ref{th:CmnCr}.
\end{enumerate}
\end{theorem}

\begin{proof}
The probability $\proba(f)$ that an expression of~$m$ clauses on~$n$ variables
computes a function~$f$
is the quotient of the number of corresponding expressions
divided by the total number of expressions
\[
  \proba(f) = 
  \frac{ m! [z^m] \phi_{\ii(f)}(z)}
    {m! n! [z^m v^n] \MG(8z,v)}.
\]
For $f = x_1 \sim \cdots \sim x_n$,
we obtain the first part of the theorem
by substitution of the expression of~$\phi_{\ii(f)}$,
derived in Proposition~\ref{prop:fgs}.
More generally, we have
\[
  \phi_{\ii(f)} (z) = 
  \prod_{\ell \geq 1} 
  \left( \ell! [ v^\ell] C(4z,v) \right)^{i(f)_\ell}.
\]
By definition, $i(f)_\ell$ is the number of
blocks of~$f$ of size~$\ell$,
so the previous equation can be rewritten
\[
  \phi_{\ii(f)} (z) = 
  \prod_{j=1}^q
  |B_j|! [v^{|B_j|}] C(4z,v),
\]
and
\begin{equation} \label{eq:phiextraction}
  m! [z^m] \phi_{\ii(f)} (z) =
  m!
  \sum_{m_1 + \cdots + m_q = m}
  \prod_{j=1}^q
  |B_j|! [z^{m_j} v^{|B_j|}] C(4z,v).
\end{equation}
The generating function~$C(z,v)$ can be expanded with respect to the excesses
\[
  C(z,v) = \sum_{r \geq -1} z^r C_r(v z),
\]
so
\begin{equation} \label{eq:Crextraction}
  |B_j|! [ z^{m_j} v^{|B_j|}] C(4z,v)
  =
  4^{m_j}
  |B_j|! [ v^{|B_j|}]
  C_{r_j}(v),
\end{equation}
where $r_j = m_j - |B_j|$.
We obtain the second part of the theorem
by combination of Equations~\eqref{eq:phiextraction} and~\eqref{eq:Crextraction}.
\end{proof}

\section{Explicit Probabilities}
\label{sec:results}

We now show on examples how Theorem \ref{th:proba-f} allows us to compute the asymptotic probability of a specific function.
Attempting to give explicit results for each and every case that may appear is not realistic; rather we  aim at giving the reader a feeling of the kind of results our method allows to obtain and the kind of technical tools we need for obtaining precise asymptotics.

We consider first a fixed Boolean function $f$ and how its probability varies when $n \rightarrow +\infty$ (i.e. when we add non-essential variables),
then turn to a family of functions that vary with $n$, either with a fixed number of blocks (this includes functions that are ``close to'' $\false$ in the sense that they have few blocks, hence few satisfying assigments), or with a number of blocks that grows with~$n$ (e.g., $\frac{n}{j}$ blocks of size $j$ for some $j\geq 2$).

\subsection{Probability of a fixed function}

We compute here the probability of any specific function, once it can be obtained, and see how it varies when $n, m \rightarrow + \infty$ with fixed ratio~$\alpha$.

\begin{proposition}
\label{prop:fixed-function}
Let $f \in \xx_n$, with $e(f)$ being the number of its essential variables, and $\ii = \ii (f)=(i_1, i_2,\dots)=(n-e(f), i_2,\dots)$ its associated integer partition. Assume $m=\alpha n \geq n - \xi( \ii(f))$; then
\[
P_{[\alpha n , n]}(f) \sim
  \frac{e^{\alpha \, e(f)}}{(2n)^{\alpha n}} \; \prod_{\ell \geq 2} 
    \left( \ell ! \phi_{\ell}\left( \frac{\alpha}{2} \right) \right)^{i_{\ell}}
    \qquad (n \rightarrow + \infty) .
\]
\end{proposition}

\begin{proof}
Let~$\vect{i}= \ii(f)$ be an integer partition with~$s(\vect{i}) = n$
and for all~$\ell \geq 2$, let~${i}_{\ell}$ be fixed, independent of~$n$.
The number of expressions with~$n$ variables and~$m$ clauses
that correspond to Boolean functions in~$\classe_\ii$ is then (cf: Proposition~\ref{prop:fgs})
\[
  n! m! [ z^m ] \frac{e^{\vect{i}_1 2 z}}{\vect{i}_1!} 
  \prod_{\ell \geq 2} \frac{\phi_{\ell}^{\vect{i}_{\ell}}(z)}{\vect{i}_{\ell}!}.
\]

We derive an asymptotic equivalent by the saddle point method for a \emph{large power} scheme, 
assuming that~$\alpha = \frac{m}{n}$ is bounded (\cite[Th.~VIII.8 p.~587]{FlajoletSedgewick}).
We get
\[
 m! \; \frac{n!}{\vect{i}_1 !}
  \left( \frac{2 e n}{m} \right)^m 
  \frac{1}{\sqrt{2 \pi m}} 
  e^{-(s(\vect{i}) - \vect{i}_1) m/n} 
  \prod_{\ell \geq 2} 
    \frac{\phi_{\ell}^{\vect{i}_{\ell}}\left( \frac{m}{2n} \right)}{\vect{i}_{\ell}!}
  ( 1 + \smallo(1) ).
\]
Using Stirling's formula, this can be rewritten as
\[
  \frac{n!}{\vect{i}_1 !}   (2n)^m 
  e^{-(s(\vect{i}) - \vect{i}_1) m/n} 
  \prod_{\ell \geq 2} 
    \frac{\phi_{\ell}^{\vect{i}_{\ell}}\left( \frac{m}{2n} \right)}{\vect{i}_{\ell}!}
  ( 1 + \smallo(1) ).
\]
By division by~$|\classe_\ii| = 2^{n - \xi(\vect{i})} \frac{n!}{ \prod_{\ell \geq 2} \vect{i}_\ell! (\ell!)^{\vect{i}_\ell}}$, we obtain the number of expressions
that correspond to any given function in~$\classe_\ii$:
\[
  2^{\xi(\vect{i}) - n} (2n)^m 
  e^{-(s(\vect{i}) - \vect{i}_1) m/n} 
  \prod_{\ell \geq 2} 
    \left( \ell ! \phi_{\ell}\left( \frac{m}{2n} \right) \right)^{\vect{i}_{\ell}}.
\]
We finally divide by the number of $(n,m)$-expressions, $4^m n^{2m}$, to obtain the asymptotic probability that a random expression with~$n$ variables and~$m$ clauses corresponds to the given Boolean function~$f$ described by 
the integer partition~$\vect{i}$:
\[
  \frac{ e^{-(s(\vect{i}) - {i}_1) m/n} }{ (2n)^m } 
  \prod_{\ell \geq 2} 
    \left( \ell ! \phi_{\ell}\left( \frac{m}{2n} \right) \right)^{\vect{i}_{\ell}}.
\]
The final form comes from the fact that $s(\ii)-i_1$ is precisely the number of essential variables of~$f$.
\end{proof}

\subsection{Asymptotics for a single-block function} 
\label{singleblock}

All Boolean variables are in a single block: We consider the class of $ x_1 \sim ... \sim x_n$ and the range $m\geq n-1$.
From Theorem~\ref{th:proba-f} , we have 
\[
\proba(x_1 \sim ... \sim x_n) = \frac{m!}{n^{2m}} . n! [ v^n ] C_{m-n}(v).
\]
We now specialize this according to the possible values for the excess $r=m-n$ and obtain the
\begin{proposition}
\begin{enumerate}
\item
For $r=-1$, we have
$\proba(x_1 \sim ... \sim x_n) =
\frac{(n-1)!}{n^{n}} \sim \sqrt{\frac{2\pi}{n}} \; e^{-n}$ .

\item
For $r=0$, we get 
$\proba(x_1 \sim ... \sim x_n)\sim \frac{\pi}{2} \; e^{-n} $.

\item
For $ r \geq 1$ but still fixed, 
$\proba(x_1 \sim ... \sim x_n) \sim C_r \, e^{-n} n^{r/2}$
where $c_r = \sqrt{2\pi} \, e^{-r} \, K_r$ with $K_r$ as in Theorem~\ref{th:connected-multigraphs-fixed-excess}.

\item
For~$r \rightarrow \infty$ and~$r = \smallo(\sqrt{n})$,
$\proba(x_1 \sim ... \sim x_n) \sim
\sqrt{\frac{3}{2}} \frac{e^{r/2}}{(2\sqrt{3})^r} e^{-n} \left(\frac{n}{r}\right)^{r/2}$.

\item
For~$r = (\alpha-1) n$ with $\alpha > 1$,
$\proba(x_1 \sim ... \sim x_n) \sim
K \left( \frac{\alpha^{\alpha-1} \cosh \zeta}{(2 \zeta)^{\alpha-1} e^\alpha} \right)^n$
where~$\zeta \coth \zeta = \alpha$
and~$K = \sqrt{\alpha} \frac{e^{2\zeta}-1-2\zeta}{\sqrt{\zeta(e^{4\zeta}-1-4\zeta e^{2\zeta})}}$.

\item
When~$r \rightarrow + \infty$ and~$2m/n - \log(n)$ is bounded, then
$$\proba(x_1 \sim ... \sim x_n) \sim
\frac{K}{(2 \zeta)^r} \left( \frac{\sinh \zeta}{\zeta} \right)^n \frac{\left(
1+r/n\right)^{n+r+1/2}}{e^{n+r}}$$
with~$\zeta$  the positive solution of~$\zeta \coth \zeta = 1 + \frac{r}{n}$
and~$K = \frac{e^{2\zeta}-1-2\zeta}{\sqrt{\zeta(e^{4\zeta}-1-4\zeta e^{2\zeta})}}$.

\item
Finally, when~$2m/n - \log(n) \rightarrow + \infty$ as~$n \rightarrow + \infty$,
$\proba(x_1 \sim ... \sim x_n) \sim \frac{1}{2^m}$
because almost all multigraphs are connected.

\end{enumerate}
\end{proposition}

\begin{proof}
We simply use the expressions for the coefficients of $C_r(v)$ given in Section~\ref{sec:asymptotics-multigraphs}.
The first three cases come from Theorem~\ref{th:connected-multigraphs-fixed-excess}, the next two cases are subcases of the sixth case which comes from Theorem~\ref{th:multigraphs-large-excess}, and the last case comes from Theorem~\ref{connected_excess-infinite}.
\end{proof}

    \subsection{Asymptotics for a two-blocks function}
\label{twoblocks}

We consider a function in the class of $x_1 \sim ... \sim x_p$, $x_{p+1} \sim ... \sim x_n$ (the block sizes are $p$ and $n-p$), which has cardinality 
$2^{n-2} \frac{n!}{p! (n-p)!}$.
We are again in range~C: $m \geq n-2$, i.e.  $r\geq -2$.
Theorem~\ref{th:proba-f} gives the generating function as
$$
\phi_\jj (z) = 
p! [ v^p ] C(4z,v) \,\cdot\, (n-p)! [ w^{n-p} ] C(4z,w),
$$
from which we readily obtain that
$$
\proba(f) =
\frac{m!}{n^{2m}} \; \sum_{d=-1}^{r+1} p! [ v^p ]  C_d(v) \,\cdot\, (n-p)! [ w^{n-p} ] C_{r-d}(w).
$$
Its asymptotics varies with the excess $r = m-n$,
and the sizes of the two blocks.
In the following propositions, we consider several cases,
depending of the respective sizes of the blocks
and the excess corresponding to the underlying multigraph.
The proofs are then presented in Sections~\ref{sec:two_blocks_large_excess} and~\ref{sec:two_blocks_fixed_excess}.

\begin{proposition}[Fixed excess and a single large part] \label{th:fixed_excess_and_a_single_large_part}
	If $p$ and $d$ belong to some fixed, finite set which does not depend on $n$, then 
\[
\proba (f) \sim K_f \, . \, n^{\frac{r+3}{2}-p} \, e^{-n},
\]
for some explicitly computable constant $K_f$.
\end{proposition}

\begin{proposition}[Fixed excess $r$ and two large parts] \label{th:fixed_excess_and_two_large_parts}
Assume that $p$ and $n-p$ both tend to infinity, as $n\to\infty$. W.l.o.g. let $p \leq n-p$.
Then we have 
\[
\proba (f) \sim 
\frac{2\pi}{e^n n^{2n+2r}} 
(n-p)^{2n+\frac{3r}{2}} \left( \frac{p}{n-p} \right)^{2p} 
\sum_{d=-1}^{r+1}  K_d \, K_{r-d} \,  
\left( \frac{p}{n-p} \right)^{\frac{3d}{2}}  
\]
for suitable constants $K_j$.
Depending on the actual growth rate of $p$ we can distinguish two cases:
\begin{enumerate}
\item
	If $p = \gamma n$ for some constant $\gamma>0$, then $p/(n-p) = \Theta(1)$ and  
\[
\proba (f) \sim K \, n^{-\frac{r+1}{2}} \, \beta^{2n}  \, e^{-n}
\qquad with \qquad
\beta = (1-\gamma)^{1-\gamma}\,\gamma^{\gamma}.
\]

\item
	If $p=\varepsilon_n \, n $ with $\varepsilon_n=o(1)$, then $p/(n-p)=o(1)$ and 
\[
\proba(f) \sim
K\,  e^{-n} \, n^{\frac{r-1}{2}} \, \varepsilon^{n\varepsilon_n-1} \,
(1-\varepsilon_n)^{(1-\varepsilon_n) n}.
\]
A more precise evaluation of probabilities gives for instance 
\begin{enumerate}
\item 
If $p=\sqrt{n}$, then $\varepsilon_n = n^{-1/2}$ and the probability of the function has order 
$ \frac{n^{-\frac{r}{2} + \frac{3}{4}} \, e^{-n-2\sqrt{n}}}{n^{\sqrt{n}}}.$

\item 
If $p= \log n$, then $\varepsilon_n = \frac{\log n}{n}$ and the probability is of order
$ \left( \frac{\log n}{n} \right)^{\log n -1}  n^{\frac{r+1}{2}} \, e^{-n}$.
\end{enumerate}
\end{enumerate}
\end{proposition}

\begin{proposition}[Large excess $r$] \label{th:two_parts_large_excess}
Assume that $r = c n$ for a fixed positive value~$c$.
Again, we distinguish two cases:
\begin{enumerate}
\item {\bf Single large part.}
If~$p$ is constant, then 
\[
\proba (f) \sim \frac{K_f}{n^{p-1}}
\left(
\frac{ (1+c)^c \cosh(\zeta) }{ e^{1+c} (2 \zeta)^c }
\right)^n, 
\]
for some explicitly computable constant $K_f$, 
where~$\zeta \coth \zeta = 1 + \frac{cn+1}{n-p}$.

\item {\bf Two proportional large parts.}
If~$p = \gamma n$ and~$r = c n$, then 
\[
\proba (f) \sim
  \frac{K_f}{n}
  \left(
    \frac{ \gamma^\gamma (1-\gamma)^{1-\gamma} (1+c)^{1+c} }{ 2^c e^{1+c }}
    g(a_0)  
  \right)^n
\]
where~$K_f$ is a computable constant,
and~$g(a_0)$ is the unique maximum  of the function in~$[0,1]$.
\begin{align*}
    g(a) 
  &=
    \left( \frac{ \cosh(\zeta_{\frac{a c}{\gamma}}(a)) }{ 1 + \frac{a c}{\gamma} } \right)^\gamma
    \left( \frac{ \cosh(\zeta_{\frac{(1-a) c}{1-\gamma}}(a)) }{ 1 + \frac{(1-a) c}{1-\gamma} } \right)^{1-\gamma}
    \left( \frac{\gamma}{\zeta_{\frac{a c}{\gamma}}(a)} \right)^{a c}
    \left( \frac{1-\gamma}{\zeta_{\frac{(1-a) c}{1-\gamma}}(a)} \right)^{(1-a) c}
\end{align*}
where~$
    \zeta_{\frac{a c}{\gamma}}(a) \coth \zeta_{\frac{a c}{\gamma}}(a) 
  = 
    1 + \frac{a c}{\gamma}
$ and~$
    \zeta_{\frac{(1-a) c}{1-\gamma}}(a) \coth \zeta_{\frac{(1-a) c}{1-\gamma}}(a) 
  = 
    1 + \frac{(1-a) c}{1-\gamma}$.
\end{enumerate}
\end{proposition}


Decomposing the two connected multigraphs according to excess gives the generating function for multigraphs with 2 connected components of respective number of vertices $p$ and $n-p$:
\begin{eqnarray*}
\phi_\jj (z)&=&
p! [v^p] \sum_{r \geq -1} (4z)^r C_r(4zv) \cdot (n-p)! [ w^{n-p} ] \sum_{s \geq -1} (4z)^s C_s(4zw)
\\ &=&
p! (n-p)! [ v^p w^{n-p} ]
\sum_{r,s \geq -1} (4z)^{r+s}\, C_r(4zv) \,  C_s(4zw)
\end{eqnarray*}
and
\begin{eqnarray*}
[z^m]\phi_\jj (z)&=&
p! (n-p)! [ z^m v^p w^{n-p} ]
\sum_{r,s \geq -1} (4z)^{r+s}\, C_r(4zv) \,  C_s(4zw)
\\ &=&
4^m \, 
p! (n-p)! [ v^p w^{n-p} ]
\sum_{r,s \geq -1, r+s+n=m}  C_r(v) \,  C_s(w)
\\ &=&
4^m \, 
p! (n-p)! [ v^p w^{n-p} ]
\sum_{r=-1}^{m-n+1}  C_r(v) \,  C_{m-n-r}(w)\\ &=&
4^m \, \sum_{r=-1}^{m-n+1} 
p! [ v^p ]  C_r(v) \cdot
(n-p)! [ w^{n-p} ] C_{m-n-r}(w).
\end{eqnarray*}
Then
\begin{eqnarray*}
\proba(f) &=& \frac{m!}{4^m n^{2m}} [z^m]\phi_\jj (z)
\\&=&
\frac{m!}{n^{2m}} \; \sum_{d=-1}^{r+1} p! [v^p]  C_d(v) \cdot (n-p)! [ w^{n-p} ] C_{r-d}(w).
\end{eqnarray*}

    \subsubsection{Function with two blocks and fixed excess} \label{sec:two_blocks_fixed_excess}

    \paragraph{Single large part}

We now present the proof of Proposition~\ref{th:fixed_excess_and_a_single_large_part}.
In the range we are working in, $p$ and $d$ belong to a fixed, finite set; let us define
\[
\gamma_{d,p} = p! [ v^p ] C_d(v).
\]
Then 
\[
\proba(f) =
\frac{m!}{n^{2m}} \; \sum_{d=-1}^{r+1} \gamma_{d,p} (n-p)! [ w^{n-p} ] C_{r-d}(w)
\]
and the asymptotic value of the coefficient $(n-p)! [ w^{n-p} ] C_{r-d}(w)$ is given by  Equation~(\ref{eq:asympt-bernhard-exces}), with a suitable constant:
\begin{eqnarray*}
(n-p)! [ w^{n-p} ] C_{r-d}(w) &\sim&
K_{r-d} .\, n^{n-p+\frac{3(r-d)-1}{2}} .
\end{eqnarray*}
We see that the dominant term of the sum $\sum_{d=-1}^{r+1} \gamma_{d,p} \left[ w^{n-p} \right] C_{r-d}(w)$ will be obtained for $d=-1$, which gives, for some suitable constant $K_f$ that can be explicitly computed
\[
\proba(f) \sim K_f \, . \, n^{\frac{r+3}{2}-p} \, e^{-n} .
\]

    \paragraph{Two large parts}

This paragraph contains the proof of Proposition~\ref{th:fixed_excess_and_two_large_parts}.
By symmetry, we can assume that $p \leq n-p$.
Recall that
\begin{eqnarray*}
\proba (f) =
\frac{m!}{n^{2m}} \; \sum_{d=-1}^{r+1} p! [v^p]  C_d(v) \cdot (n-p)! [ w^{n-p} ] C_{r-d}(w),
\end{eqnarray*}
but now the coefficients $\left[ v^p \right]  C_d(v)$ and $ \left[ w^{n-p} \right] C_{r-d}$ can both be obtained from the expansion~(\ref{eq:asympt-bernhard-exces}) ($p$ \emph{and} $n-p$ are large); moreover we are dealing with a fixed number of terms:
\begin{eqnarray*}
\proba(f) &\sim &
\frac{m!}{n^{2m}} \; \sum_{d=-1}^{r+1} K_d \, p^{p+\frac{3d-1}{2}} K_{r-d} \, (n-p)^{n-p+\frac{3(r-d)-1}{2}}
\\ &\sim &
\frac{m!}{n^{2m}} \;  p^{p-\frac{1}{2}} (n-p)^{n-p+\frac{3r-1}{2}}   
\sum_{d=-1}^{r+1}  K_d \, K_{r-d} \,  
\left( \frac{p}{n-p} \right)^{\frac{3d}{2}} 
\\ &\sim &
\sqrt{\frac{2\pi \,n}{p(n-p)}} \,\frac{e^{-n}}{n^{n+r}} \, (n-p)^{n+\frac{3r}{2}} \, \left( \frac{p}{n-p} \right)^{p} 
\sum_{d=-1}^{r+1}  K_d \, K_{r-d} \,  
\left( \frac{p}{n-p} \right)^{\frac{3d}{2}}  .
\end{eqnarray*}
Now we have to find the behaviour of the sum in the above expression, and we see that there are two different cases:
\begin{enumerate}
\item
If $p$ and $n$ are proportional, then $p/(n-p) = \Theta(1)$ (for simplification we set $p = \gamma n$ and assume $\gamma$ is constant, but the sequel only requires that $\gamma = \Theta(1)$);
all terms  $\left( \frac{p}{n-p} \right)^{\frac{3d}{2}}$ contribute to a constant factor, and the sum itself is constant, hence for a suitable constant \footnote{Here and in what follows, the constant denoted by $K$ may vary and may depend on~$r$ -- but it is always possible to get an explicit, though cumbersome, expression for it.} $K$ we have 
\[
\proba(f) \sim K \, n^{\frac{r-1}{2}} \, \beta^{n}  \, e^{-n}
\qquad with \qquad
\beta = (1-\gamma)^{1-\gamma}\,\gamma^{\gamma}.
\]

\item
If $p/(n-p)=o(1)$ i.e. $p=o(n)$, then the first term of the sum dominates: Up to a constant
multiplicative factor, the whole sum is asymptotically equivalent to $\left( \frac{n-p}{p}
\right)^{\frac{3}{2}}$. Setting $\varepsilon = p/n$ we get
\begin{eqnarray*}
\proba(f) &\sim &
K\,  e^{-n} \, n^{\frac{r-1}{2}} \,\varepsilon^{n\varepsilon-1} \, (1-\varepsilon)^{(1-\varepsilon) n}  .
\end{eqnarray*}

\end{enumerate}

    \subsubsection{Large excess} \label{sec:two_blocks_large_excess}

This section contains the proof of Proposition~\ref{th:two_parts_large_excess}.
Let~$C_{n+r,n}$ denote the number of connected multigraphs
with~$n$ vertices and excess~$r$.
For this proof, we rewrite the asymptotics of $C_{n+r,n}$
when~$r \rightarrow \infty$ and~$(r+n) e^{-2 r/n} \rightarrow \infty$,
already derived in Theorem~\ref{th:multigraphs-large-excess}, as
\begin{equation} \label{eq:cmg}
  C_{n+r,n} 
  =
  \frac{ \alpha(\zeta_{\frac{r}{n}}) }{ \sqrt{2\pi} (2 \zeta_{\frac{r}{n}})^r } 
  \left( 
    \frac{\cosh \zeta_{\frac{r}{n}}}{1 + \frac{r}{n}} 
  \right)^n 
  n^{n+r-\frac{1}{2}}
  \left(1 + \bigO \left( r e^{-2r/n} \right)^{-\frac{1}{2}+\epsilon} \right)
\end{equation}
for any small~$\epsilon > 0$,
where~$\zeta_{\frac{r}{n}} \coth \zeta_{\frac{r}{n}} = 1+\frac{r}{n}$
and~$\alpha(\zeta) = \frac{e^{2\zeta}-1-2\zeta}{\sqrt{\zeta (e^{4\zeta}-1-4\zeta e^{2\zeta})}}$.

We are interested here in the probability that
a random~$2$-Xor expression
with~$n$ variables and~$m$ clauses
compute the Boolean function
with two blocks of sizes~$\gamma n$ and~$(1-\gamma) n$
\[
  x_1 \sim \ldots \sim x_{\gamma n},\ x_{\gamma n + 1} \sim \ldots \sim x_n.
\]
This probability can be expressed as
\begin{align*}
    \Ptwoblocks
  &=
    \frac{m!}{n^{2m}} \sum_{d=-1}^{r+1} C_{\gamma n + d, \gamma n} C_{(1-\gamma)n + r-d, (1-\gamma)n}
  \\
  &=
    \frac{m!}{n^{2m}} \sum_{d=-1}^{r+1} A_d
\end{align*}
where~$r = m-n$ is the global excess of the multigraphs
representing the random expression
and $d$ (resp.~$r-d$) the excess of its first (resp. second) connected component.


The main ingredient for the proof of Proposition~\ref{th:two_parts_large_excess}
is the Laplace methode. 
It involves first a reduction to a problem of real analysis, 
then the analysis of a real function.
Those steps are detailed in the next two paragraphs.

    \paragraph{Reduction to a real analysis problem}

We make the assumption that the excess~$r$ increases proportionately to~$n$,
so~$r = (\alpha-1) n$ where~$\alpha=\frac{m}{n} > 0$ is a constant.
In that case,
\begin{align*}
    \frac{m!}{n^{2m}}
  &=
    \frac{(n+r)!}{n^{2(n+r)}}
  \\
  &\sim
    \frac{(n+r)^{n+r} \sqrt{2 \pi (n+r)}}{n^{2(n+r)} e^{n+r}}
  \\
  &\sim
    \frac{\left( 1+\frac{r}{n} \right)^{n+r+1/2} \sqrt{2 \pi n}}{n^{n+r} e^{n+r}}
  \\
  &\sim
    \sqrt{2\pi} \, \frac{\alpha^{\alpha n + 1/2}}{e^{\alpha n}} n^{-\alpha n + 1/2}
\end{align*}

Let us summarize some notations\\
\begin{center}
\begin{tabular}{ll}
  \text{total number of vertices} & $n$ \\
  \text{size of the first and smallest block} & $p = \gamma n$ \\
  \text{size of the second block} & $n-p = (1-\gamma) n$ \\
  \text{total excess} & $r = c n$ \\
  \text{excess of the first block} & $d = a r$ \\
  \text{excess of the second block} & $r-d = (1-a) r$
\end{tabular}
\end{center}

The expression of~$A_d$ is quite complicated,
so, in order to avoid forgetting some terms in the product,
we write them down in the following array
\[
\begin{array}{cc}
C_{\gamma n + a r, \gamma n} & 
C_{(1-\gamma)n + (1-a) r, (1-\gamma)n}
\\ \hline
\gamma n & 
(1-\gamma)n 
\\
a r & 
(1-a) r 
\\ \hline
\alpha(\zeta_{\frac{a c}{\gamma}}) & 
\alpha(\zeta_{\frac{(1-a) c}{1-\gamma}})
\\
\cosh(\zeta_{\frac{a c}{\gamma}})^{\gamma n} & 
\cosh(\zeta_{\frac{(1-a) c}{1-\gamma}})^{(1-\gamma) n}                                        
\\
\scriptstyle \gamma^{(\gamma+a c) n - 1/2} n^{(\gamma+a c) n -
1/2} & 
\scriptstyle (1-\gamma)^{(1-\gamma+(1-a) c) n - 1/2} n^{(1-\gamma+(1-a) c) n - 1/2} 
\\
2^{a c n} \zeta_{\frac{a c}{\gamma}}^{a c n} & 
2^{(1-a) c n} \zeta_{\frac{(1-a) c}{1-\gamma}}^{(1-a) c n}
\\
\left(1+\frac{a c}{\gamma}\right)^{\gamma n} & \left(1+\frac{(1-a) c}{1-\gamma}\right)^{(1-\gamma) n}
\end{array}.
\]

\bigskip
Now write $A_d = C_{\gamma n + a r, \gamma n} C_{(1-\gamma)n + (1-a) r, (1-\gamma)n}$ as
\begin{eqnarray*}
    A_d &\sim&
    \frac{\left( \gamma^{\gamma} (1-\gamma)^{1-\gamma} \right)^n n^{(c+1) n - 1}}{2\pi \sqrt{\gamma(1-\gamma)} 2^{c n}}
    \alpha(\zeta_{\frac{a c}{\gamma}}) \alpha(\zeta_{\frac{(1-a) c}{1-\gamma}})
   \\ && \qquad      \times 
    \left(
      \frac{ \cosh(\zeta_{\frac{a c}{\gamma}})^{\gamma} \cosh(\zeta_{\frac{(1-a) c}{1-\gamma}})^{1-\gamma} }{ \left( 1 + \frac{ a c}{\gamma} \right)^\gamma   \left( 1 + \frac{(1-a) c}{1-\gamma} \right)^{1-\gamma} }
      \left( \frac{\gamma}{\zeta_{\frac{a c}{\gamma}}} \right)^{a c}
      \left( \frac{1-\gamma}{\zeta_{\frac{(1-a) c}{1-\gamma}}} \right)^{(1-a)c}
    \right)^n
\end{eqnarray*}
which gives
\begin{align*}
    \frac{m!}{n^{2m}} A_d
  &\sim
    \frac{\left( \gamma^{\gamma} (1-\gamma)^{1-\gamma} \right)^n 
    (c+1)^{(c+1) n + \frac{1}{2}}}{\sqrt{2 \pi n} \sqrt{ \gamma (1-\gamma)} 2^{c n} e^{(c+1) n}}
    \alpha(\zeta_{\frac{a c}{\gamma}}) \alpha(\zeta_{\frac{(1-a) c}{1-\gamma}})
    g(a)^n  
  \\
  &\sim
    \sqrt{\frac{c+1}{\gamma(1-\gamma)}}
    \left(
      \frac{ \gamma^\gamma (1-\gamma)^{1-\gamma} (c+1)^{(c+1)} }{ 2^{c} e^{c+1}}
    \right)^n
    \frac{\alpha(\zeta_{\frac{a c}{\gamma}}) \alpha(\zeta_{\frac{(1-a) c}{1-\gamma}})}{\sqrt{2 \pi n}}
    g(a)^n
\end{align*}
where
\begin{align*}
    g(a) 
  &=
    \left( \frac{ \cosh(\zeta_{\frac{a c}{\gamma}}) }{ 1 + \frac{a c}{\gamma} } \right)^\gamma
    \left( \frac{ \cosh(\zeta_{\frac{(1-a) c}{1-\gamma}}) }{ 1 + \frac{(1-a) c}{1-\gamma} } \right)^{1-\gamma}
    \left( \frac{\gamma}{\zeta_{\frac{a c}{\gamma}}} \right)^{a c}
    \left( \frac{1-\gamma}{\zeta_{\frac{(1-a) c}{1-\gamma}}} \right)^{(1-a) c},
  \\
    \zeta_{\frac{a c}{\gamma}} \coth \zeta_{\frac{a c}{\gamma}} 
  &= 
    1 + \frac{a (\alpha-1)}{\gamma}, \qquad \alpha(\zeta_{\frac{a c}{\gamma}}) =
    \frac{e^{2\zeta_{\frac{a c}{\gamma}}}-1-2\zeta_{\frac{a c}{\gamma}}}{\sqrt{\zeta_{\frac{a c}{\gamma}} (e^{4\zeta_{\frac{a c}{\gamma}}}-1-4\zeta_{\frac{a c}{\gamma}} e^{2\zeta_{\frac{a c}{\gamma}}})}},
  \\
    \zeta_{\frac{(1-a) c}{1-\gamma}} \coth \zeta_{\frac{(1-a) c}{1-\gamma}} 
  &= 
    1 + \frac{(1-a) c}{1-\gamma}, \qquad \alpha(\zeta_{\frac{(1-a) c}{1-\gamma}}) =
    \frac{e^{2\zeta_{\frac{(1-a) c}{1-\gamma}}}-1-2\zeta_{\frac{(1-a) c}{1-\gamma}}}{\sqrt{\zeta_{\frac{(1-a) c}{1-\gamma}} (e^{4\zeta_{\frac{(1-a) c}{1-\gamma}}}-1-4\zeta_{\frac{(1-a) c}{1-\gamma}} e^{2\zeta_{\frac{(1-a) c}{1-\gamma}}})}}.
\end{align*}

We will see in the next paragraph that 
the dominant part of the sum~$\sum_{d=-1}^{r+1} A_d$
is reached for a compact range of~$a$ included in~$]0,1[$.
This justifies the use of the asymptotic formula~\ref{eq:cmg}.
Furthermore, the error term of~$A_d$
\[
  \left( 1 + \bigO \left( a r e^{- a c/\gamma} \right)^{-\frac{1}{2}+\epsilon} \right)
  \left( 1 + \bigO \left( (1-a) r e^{-(1-a) c / (1-\gamma)} \right)^{-\frac{1}{2}+\epsilon} \right)
\]
becomes uniform in~$a$, so
\[
    \Ptwoblocks
  \sim 
    \sqrt{\frac{c+1}{\gamma(1-\gamma)}}
    \left(
      \frac	{ \gamma^\gamma (1-\gamma)^{1-\gamma} (c+1)^{c+1} }
      		{ 2^{c} e^{c+1}}
    \right)^n
    \frac{1}{\sqrt{2 \pi n}}
    \sum_{d=0}^r
      \alpha(\zeta_{\frac{a c}{\gamma}}) \alpha(\zeta_{\frac{(1-a) c}{1-\gamma}})
      g \left( \frac{d}{r} \right)^n.
\]

    \paragraph{Analysis of~$g(a)$}

We prove here that~$g(a)$ has a unique maximum~$a_0$ in~$[0,1]$
such that~$0 < a_0 < 1$.
To do so, we use the concavity of~$\log(g(a))$.
The \emph{Laplace's method for sums}
described in~\cite{FlajoletSedgewick} p.761
then leads to
\begin{align*}
    \sum_{d=0}^r
      \alpha(\zeta_{\frac{a c}{\gamma}}) \alpha(\zeta_{\frac{(1-a) c}{1-\gamma}})
      g \left( \frac{d}{r} \right)^n
  &\sim
    \sqrt{\frac{2 \pi}{\lambda n}} 
    \alpha(\zeta_{\frac{a_0 c}{\gamma}}) \alpha(\zeta_{\frac{(1-a_0) c}{1-\gamma}}) 
    g(a_0)^n
\end{align*}
where~$\lambda = - \frac{g''(a_0)}{g(a_0)}$, so
\begin{align*}
  \Ptwoblocks
&\sim
  \sqrt{\frac{c+1}{\gamma(1-\gamma) \lambda}}
  \left(
    \frac{ \gamma^\gamma (1-\gamma)^{1-\gamma} (c+1)^{c+1} }{ 2^c e^{c+1}}
  \right)^n
  \frac{\alpha(\zeta_{\frac{a c}{\gamma}})(a_0) \alpha(\zeta_{\frac{(1-a) c}{1-\gamma}})(a_0)}{n}
  g(a_0)^n  .
\end{align*}


The proof of the asymptotics is now reduced
to a real analysis problem: Proving that
\begin{align*}
    g(a) 
  &=
    \left( \frac{ \cosh(\zeta_{\frac{a c}{\gamma}}) }{ 1 + \frac{a c}{\gamma} } \right)^\gamma
    \left( \frac{ \cosh(\zeta_{\frac{(1-a) c}{1-\gamma}}) }{ 1 + \frac{(1-a) c}{1-\gamma} } \right)^{1-\gamma}
    \left( \frac{\gamma}{\zeta_{\frac{a c}{\gamma}}} \right)^{a c}
    \left( \frac{1-\gamma}{\zeta_{\frac{(1-a) c}{1-\gamma}}} \right)^{(1-a) c}
\\
&=
  \left( 
    \frac{ \cosh \zeta_{\frac{a c}{\gamma}} }{ \zeta_{\frac{a c}{\gamma}}^{x_1} }
    \frac{ \gamma^{x_1} }{ 1+x_1 }
  \right)^\gamma
  \left(
    \frac{ \cosh \zeta_{\frac{(1-a) c}{1-\gamma}} }{ \zeta_{\frac{(1-a) c}{1-\gamma}}^{x_2} }
    \frac{ (1-\gamma)^{x_2} }{ 1+x_2 }
  \right)^{1-\gamma},
\end{align*}
where~$x_1 = \frac{ a c }{ \gamma }$ and~$x_2 = \frac{ (1-a) c }{ 1-\gamma }$,
has a unique maximum in the interior of~$]0,1[$ 
for all~$c > 0$ and~$\gamma \in ]0, 1/2]$.
Let~$\zeta(x)$ be defined implicitly as
\[
  \zeta \coth \zeta = 1 + x,
\]
then
\begin{align*}
  \frac{\zeta'}{\zeta} &= \frac{1}{\zeta^2 - x (1+x)}, \\
  \zeta' \tanh \zeta &= \frac{\zeta^2}{(\zeta^2 - x (1+x))(1+x)},
\end{align*}
so
\begin{align*}
    \frac{d}{d x} \log 
    \left( 
      \frac{ \cosh \zeta}{ \zeta^{x} }
      \frac{ \gamma^{x} }{ 1+x }
    \right)
  &=
    \zeta' \tanh(\zeta) - x \frac{\zeta'}{\zeta} - \frac{1}{1+x} + \log(\gamma) - \log(\zeta)
  \\
  &=
    \frac{\zeta^2}{(\zeta^2 - x (1+x))(1+x)} - \frac{x}{\zeta^2 - x (1+x)} - \frac{1}{1+x} + \log \left( \frac{\gamma}{\zeta} \right)
  \\
  &=
    \frac{\zeta^2 - x(1+x)}{(\zeta^2 - x (1+x))(1+x)} - \frac{1}{1+x} + \log \left( \frac{\gamma}{\zeta} \right)
  \\
  &=
    \log \left( \frac{\gamma}{\zeta} \right)
\end{align*}
and
\begin{align*}
    \frac{d}{d x} \log 
    \left( 
      \frac{ \cosh \zeta}{ \zeta^{x} }
      \frac{ (1-\gamma)^{x} }{ 1+x }
    \right)
  &=
    \log \left( \frac{1-\gamma}{\zeta} \right).
\end{align*}
Therefore,
\begin{align*}
  \frac{d}{d a} \log(g(a))
&=
  \gamma  
  \left( \frac{d}{d a} x_1 \right)
  \frac{d}{d x_1} \log 
    \left( 
      \frac{ \cosh \zeta_{\frac{a c}{\gamma}}}{ \zeta_{\frac{a c}{\gamma}}^{x_1} }
      \frac{ \gamma^{x_1} }{ 1+x_1 }
    \right)\\
& \hspace{1.5cm}
  +  
  (1-\gamma)
  \left( \frac{d}{d a} x_2 \right)
  \frac{d}{d x_2} \log 
    \left( 
      \frac{ \cosh \zeta_{\frac{(1-a) c}{1-\gamma}}}{ \zeta_{\frac{(1-a) c}{1-\gamma}}^{x_2} }
      \frac{ (1-\gamma)^{x_2} }{ 1+x_2 }
    \right)
\\
&=
  c \log \left( \frac{\gamma}{\zeta_{\frac{a c}{\gamma}}} \right) - c \log \left( \frac{1-\gamma}{\zeta_{\frac{(1-a) c}{1-\gamma}}} \right)
\\
&=
  c \log \left( \frac{\gamma}{1-\gamma} \frac{\zeta_{\frac{(1-a) c}{1-\gamma}}}{\zeta_{\frac{a c}{\gamma}}} \right)
\end{align*}
and
\begin{align*}
  \frac{1}{c} \frac{d^2}{(d a)^2} \log(g(a))
&=
  \frac{d}{d a} \log \left( \frac{\zeta_{\frac{(1-a) c}{1-\gamma}}}{\zeta_{\frac{a c}{\gamma}}} \right)
\\
&=
  \left( \frac{d}{d a} x_2 \right)
  \frac{ \zeta_{\frac{(1-a) c}{1-\gamma}}' }{ \zeta_{\frac{(1-a) c}{1-\gamma}} }
  -
  \left( \frac{d}{d a} x_1 \right)
  \frac{ \zeta_{\frac{a c}{\gamma}}' }{ \zeta_{\frac{a c}{\gamma}} }
\\
&=
  - \frac{c}{1-\gamma} 
  \frac{1}{ \zeta_{\frac{(1-a) c}{1-\gamma}}^2 - x_2 (1+x_2) }
  -
  \frac{c}{\gamma}
  \frac{1}{ \zeta_{\frac{a c}{\gamma}}^2 - x_1 (1+x_1) }
\end{align*}
which is negative because for all~$x>0$,
\[
  \zeta(x) > \sqrt{x(1+x)}.
\]
Therefore, $\frac{d}{d a} \log(g(a))$ is decreasing on~$]0,1[$.
Let us summarize some values:
\[
\begin{array}{c|ccc}
a & 0 & \gamma & 1\\ \hline
x_1 = \frac{a c}{\gamma} & 0 & c & \frac{c}{\gamma} \\
x_2 = \frac{(1-a) c}{1-\gamma} & \frac{c}{1-\gamma} & c & 0\\
\zeta_{\frac{a c}{\gamma}} & 0 & \zeta(c) & \zeta \left(\frac{c}{\gamma} \right)\\
\zeta_{\frac{(1-a) c}{1-\gamma}} & \zeta \left( \frac{c}{1-\gamma} \right) & \zeta(c) & 0\\
\frac{\zeta_{\frac{(1-a) c}{1-\gamma}}}{\zeta_{\frac{a c}{\gamma}}} & +\infty & & -\infty\\
\frac{d}{d a} \log(g(a)) & + \infty & & -\infty
\end{array}
\]
so $\frac{d}{d a} \log(g(a))$ has a zero on~$]0,1[$,
and~$g(a)$ has a unique maximum in~$]0,1[$.

    \subsection{Number of blocks proportional to $n$}

A general approach via Theorem~\ref{th:proba-f} seems difficult, so we assume a certain regularity: 
Let $f$ denote a Boolean function such that the associated integer partition is of the form  
$\mathbf{i}(f)=(0,\dots,0,n/g,0,\dots)$, with $g\ge 2$. 
Note that the corresponding multigraph has to have at least $m=(g-1)\cdot\frac{n}{g}$ edges. Thus, in contrary to the previously discussed cases, the excess $r = - \frac{n}{g}$ is no more bounded from below as $n\to\infty$. Such functions may now appear even close to the threshold $1/2$. 
In Proposition~\ref{th:exact_proportional}, we derive an exact result for those functions; an asymptotic result is stated in Proposition~\ref{th:asympt_proportional}.

\begin{proposition} \label{th:exact_proportional}
The number of expressions $\N(f)$ with $n$ variables and $m$ clauses computing a function $f$ with
associated integer partition representation of the form $\mathbf{i}(f)=(0,\dots,0,n/g,0,\dots)$,
i.e. $n/g$ blocks of size $g$, is
given by
\begin{equation} \label{explicit_form}
\N(f)= m!4^m (g!)^{\frac{n}{g}}[z^m]\Big(\sum_{j=1}^{g}\frac{(-1)^{j-1}}j  e_{j,g-j}(z)\Big)^{\frac{n}g}
\end{equation} 
with 
\[
e_{j,n}(z)
=\sum_{\substack{\sum_{\ell=1}^{j}k_\ell=n\\k_\ell\ge 0}}
\binom{n}{k_1,\dots,k_j}\frac{\exp\(\sum_{\ell=1}^j\frac{(k_\ell+1)^2z}2\)}{\prod_{r=1}^{j} (k_\ell+1)!}.
\]
\end{proposition}

\begin{remark}
One might be tempted to use again Theorem~\ref{th:proba-f}. For Boolean functions $f$
having associated integer partition of the form $\mathbf{i}(f)=(0,\dots,0,n/g,0,\dots)$ with 
$g\ge 2$ this yields 
\[
\N (f) = m!4^m  
\sum_{\substack{\sum_{k=1}^{q}r_k=m-n\\r_k\ge -1}}
|B_1|! \cdots |B_q|! \left[ v^{|B_1|} \dots v^{|B_q|} \right]
\prod_{j=1}^{q}C_{r_j}(v_j),
\]
where $B_1, \ldots B_q$ are the blocks of the set partition, or equivalently the components of the Boolean function, with respective excesses $r_1, \ldots , r_q$.
Here the number of blocks is  $q=\frac{n}g$ and all of them have size~$g$; hence 
\[
\N(f) = 4^m 
(g!)^{\frac{n}{g}}\sum_{\substack{\sum_{k=1}^{q}r_k=m-n\\r_k\ge
-1}}\prod_{j=1}^{\frac{n}{g}} [v^{g} ]C_{r_j}(v).
\]
However, it seems to get enough information on $C_{r_j}(v)$ to derive expression
\eqref{explicit_form} from this formula. 
\end{remark}

\begin{proof}
Instead of analyzing the coefficients $C_{r}(v)$ of $C(z,v)$ we use directly the relation 
$C(z,v)=\log \MG(z,v)$.
Since $\mathbf{i}(f)=(0,\dots,0,n/g,0,\dots)$, with $g\ge 2$, we have
\begin{eqnarray*}
\N(f) &=& m![z^m]\phi_{\ii}(z)
\\ &=& 
m![z^m]\prod_{\ell\ge 1} \big( \ell! [ v^\ell ] C(4z,v) \big)^{i_\ell}
\\ &=&
m!4^m(g!)^{\frac{n}{g}}[z^m]\Big([v^g]\log \MG(z,v)\Big)^{\frac{n}g}.
\end{eqnarray*}
Let $$\mghat(z,v)=(\MG(z,v)-1)/v=\sum_{n\ge 0}e^{\frac{(n+1)^2z}2}\frac{v^n}{(n+1)!},$$ such that
\[
\log \MG(z,v)=\sum_{j\ge 1}\frac{(-1)^{j-1}}j v^j\mghat^j(z,v).
\]
We get
\begin{align*}
\N(f) 
&= m! [z^m]\prod_{\ell\ge 1} \big([\frac{v^\ell}{\ell!}]C(4z,v)\big)^{i_\ell}
=   m!4^m(g!)^{\frac{n}{g}}[z^m]\Big([v^g]\log \MG(z,v)\Big)^{\frac{n}g}
\\
&=  m!4^m(g!)^{\frac{n}{g}}[z^m]\Big(\sum_{j=1}^{g}\frac{(-1)^{j-1}}j[v^{g-j}]\mghat^j(z,v)\Big)^{\frac{n}g}.
\end{align*}
We can expand $\mghat^j(z,v)$ in terms of the functions $e_{j,n}(z)$ as defined above using the multinomial theorem:
\[
\mghat^j(z,v)=\bigg(\sum_{n\ge 0}e^{\frac{(n+1)^2z}2}\frac{v^n}{(n+1)!}\bigg)^j
= \sum_{n\ge 0}e_{j,n}(z)v^n. 
\]
Extraction of coefficients then directly leads to the stated result.
\end{proof}

\begin{coroll}
Under the assumptions of Proposition~\ref{th:exact_proportional}, in the case $g=2$ we get
\[
\proba(f)
=\frac{1}{n^{2m}}\sum_{\ell=0}^{\frac{n}2}\binom{\frac{n}2}\ell \big(\ell+\frac{n}2\big)^{m}(-1)^{\frac{n}2-\ell},
\]
and for $g=3$
\[
\proba(f)
=\frac{1}{n^{2m}}\sum_{\ell=0}^{\frac{n}3}\sum_{j=0}^\ell \binom{\frac{n}3}\ell \binom{\ell}j (\frac{n}2+\ell+2j)^{m}(-3)^{\ell-j} 2^{\frac{n}3-\ell}.
\]
\end{coroll}

\begin{proof}
Using Proposition~\ref{th:exact_proportional}, Equation~\eqref{explicit_form}, we obtain first
\begin{align*} 
	\N(f)&= m!4^m (2!)^{\frac{n}{2}}[z^m]\Big(\frac{1}{2}e^{2z}-\frac12 e^z\Big)^{\frac{n}2}
=m!4^m[z^m]\Big(e^{2z}-e^z\Big)^{\frac{n}2} \\
&=m!4^m[z^m]e^{\frac{zn}2}\Big(e^{z}-1\Big)^{\frac{n}2}.
\end{align*} 
The expansion of $\Big(e^{z}-1\Big)^{\frac{n}2}$ by the binomial theorem and the extraction of coefficients
leads then to the stated result after dividing by the total number of expressions $(4n^2)^m$. We
proceed for $g=3$ in a similar way: 
\begin{equation*} 
	\N(f)= m!4^m (3!)^{\frac{n}{3}}[z^m]\left(\frac{1}{6} e^{\frac{9z}2}-\frac12
	e^{\frac{5z}2}+\frac13e^{\frac{3z}2}\right)^{\frac{n}3}.
\end{equation*} 
In order to extract coefficients we use
\[
	\left(\frac{1}{6}e^{\frac{9z}2}-\frac12 
	e^{\frac{5z}2}+\frac13 e^{\frac{3z}2}\right)^{\frac{n}3}
=e^{\frac{nz}2}\Big(\frac{1}{6} e^{3z}-\frac12 e^{z}+\frac13 \Big)^{\frac{n}3},
\]
and expand twice using the binomial theorem. This leads to the stated result after a few elementary computations.
\end{proof}

When considering asymptotics, we observed in the Sections~\ref{singleblock} and~\ref{twoblocks}
that the asymptotic behaviour is different depending on the fact whether the excess is constant or
large, \emph{i.e.} tending to infinity. For functions with two blocks there are also several
phases in the case of large excess. But this observation is misleading, because in fact is not the
excess but rather the distance from the minimal possible excess which determines the behaviour. In
the case considered in this section, we will therefore write $m=\frac{g-1}{g}\cdot n +\kappa_n$,
with $\kappa_n\ge 0$, because the minimal excess is $-n/g$. Furthermore, it turns out that there
is no qualitative difference between constant and large $\kappa_n$ in the sense that both cases
can be covered by one single formula. This holds, however, only up to the interesting range, which
has been shown to be $\kappa_n=\Theta(n^{2/3})$ in \cite{CrDa04}. 

The expression for $\N(f)$ given in Equation~\eqref{explicit_form} that appears is a fixed
function $G(z)=[v^g]\log \MG(z,v)$ raised to a large power $n/g$; e.g., for $g=2$ we have
$G(z)=\frac12e^{2z}-\frac12 e^z$ and for $g=3$ we have $G(z)=\frac{1}{6}e^{\frac{9z}2}-\frac12
	e^{\frac{5z}2}+\frac13e^{\frac{3z}2}$. 
By definition of $\log \MG(z,v)=\sum_{r\ge -1}z^rC_r(zv)$, the function $G(z)$ is of the form
$G(z)=\sum_{\ell\ge g-1}a_\ell z^\ell$ for certain coefficients $a_\ell$.
Thus, in case of constant $\kappa_n$, Equation~\eqref{explicit_form} involves a sum with a bounded range
depending only on $\kappa_n$:
\begin{align*}
	\N(f)&= m!4^m (g!)^{\frac{n}{g}}[z^m]G(z)^{\frac{n}g}
= m!4^m(g!)^{\frac{n}{g}} [z^{\frac{g-1}{g}\cdot n +\kappa_n}]\left(\sum_{\ell\ge g-1}a_\ell
z^\ell\right)^{\frac{n}g} \\
&=m!4^m(g!)^{\frac{n}{g}} [z^{\kappa_n}]\left(\sum_{\ell\ge 0}\tilde{a}_{\ell} z^\ell\right)^{\frac{n}g},
\end{align*} 
with $\tilde{a}_\ell=a_{g-1+\ell}$ for $\ell\ge 0$.
Using
\begin{equation*}
	\left(\sum_{\ell\ge 0}\tilde{a}_{\ell} z^\ell\right)^{\frac{n}g} 
= \sum_{i\ge 0}z^i\sum_{\substack{k_j\ge0\\ \sum_{j=1}^{\frac{n}g}k_j=i}}\binom{n/g}{k_1,\dots,k_{n/g}}\prod_{s=1}^{n/g} \tilde{a}_{k_s} 
\end{equation*}
we get
\[
\N(f)= m!4^m (g!)^{\frac{n}{g}}\sum_{\substack{k_j\ge 0\\ \sum_{j=1}^{\frac{n}g}k_j=\kappa_n}}
\binom{n/g}{k_1,\dots,k_{n/g}}\prod_{s=1}^{n/g} \tilde{a}_{k_s},
\]
with $\tilde{a}_\ell$ denoting the shifting coefficients of $G(z)=[v^g]\log \MG(z,v)$.

For $\kappa_n\to\infty$ the saddle point method applies and we can compute $\N(f)$ asymptotically,
though the expressions quickly become messy as $g$ grows. For $g=2$ we obtain the following
result: 

\begin{proposition} \label{th:asympt_proportional}
The number of expressions $\N(f)$ with $n$ variables and $m$ clauses computing a function $f$ with
associated integer partition representation of the form $\mathbf{i}(f)=(0,n/2,0,0,\dots)$,
i.e. $n/2$ blocks of size $2$, is given for $m= \frac{n}{2} + \kappa_n$ with $\kappa_n = O (n^ {2/3} )$ by 
\[
\N(f)=m!\frac{2^{2m+\frac n2+1}}{\sqrt{6\pi ns_n}} 
s_n^{-m+\frac{n}2}\exp\(\frac{3ns_n}4+\frac1 {48}n s_n^2+O(ns_n^4)\).
\]
where $s_n$ is the unique positive solution of 
$
\frac{z(2e^{z}-1)}{e^{z}-1}=1+\frac{2\kappa_n}{n},
$
and satisfies
\[
s_n=\frac{4}{3}\cdot\frac{\kappa_n}{n}+\mathcal{O}\left(\frac{\kappa_n^2}{n^2}\right).
\]
\end{proposition}

\begin{proof}
In the expression for $\N(f)$, Eq.~\eqref{explicit_form}, a fixed function 
\[
G(z)=\sum_{j=1}^{g}\frac{(-1)^{j-1}}j e_{j,g-j}(z)
\]
raised to a large power appears. Hence, we can apply the saddle-point technique to obtain an asymptotic expansion of $\N(f)$ for $m$ and $n$ tending to infinity.
In general,
\begin{eqnarray*}
\N(f) &=& 
m!4^m(g!)^{\frac{n}{g}}[z^m]\big(G(z)\big)^{\frac{n}g} 
\\ &=& 
\frac{m!4^m(g!)^{\frac{n}{g}}}{2\pi i } \oint_r \frac{G^{\frac{n}g}(z)}{z^{m+1}}dz
\\ &=& 
\frac{m!4^m(g!)^{\frac{n}{g}}}{2\pi i } \oint_r \exp\big( \frac{n}{g}\log G(z) - (m+1)\log z \big) dz.
\end{eqnarray*}
The saddle point equation is given by
\[
\frac{zG'(z)}{G(z)}=\frac{m+1}{\frac{n}g}.
\]
By our previous observation on functions $f$ with associated integer partition representation of
the form $\mathbf{i}(f)=(0,\dots,0,n/g,0,\dots)$ we must have $m\ge \frac{g-1}{g}\cdot n$ in order
to ensure that  $\N(f)>0$.
Hence, we assume that $m=\frac{g-1}{g}\cdot n +\kappa_n-1$, with $\kappa_n\ge 1$ and asympotically
$\kappa_n=o(n)$.\footnote{It is also possible to extend the analysis to larger $m$, i.e. $m\sim
\alpha\cdot n$ with $\alpha>\frac{g-1}{g}$, or $m\gg n$.}
Thus, we obtain further
\[
\frac{zG'(z)}{G(z)}=g-1 + g\frac{\kappa_n}{n}.
\]
For every concrete fixed $g$ it should be possible to treat this equation (preferentially using a computer algebra system); we outline the main steps for the simplest case $g=2$ and the case of $\kappa_n\to\infty$, assuming that $\kappa_n=\mathcal{O}(n^{\frac23})$. For
$g=2$ we get $G(z)=\frac12 e^{z}\cdot (e^{z}-1)$. It is convenient to cancel the factor $\frac12$,
appearing in $G(z)$ (and which is then raised to the power $\frac{n}2$) with 
$(2!)^{\frac{n}{2}}$. We define $\tilde{G}(z)= e^{z}\cdot (e^{z}-1)$ such that the saddle point equation for $\tilde{G}(z)$ is identical to the previous equation for $G(z)$. We obtain
\begin{align*}
\N(f) &=  \frac{m!4^m}{2\pi i } \oint_r \exp\big( \frac{n}2\log \tilde{G}(z) - (m+1)\log z
\big)dz\\
&=  \frac{m!4^m}{2\pi i } \oint_r \exp\big( \frac{n}{2}z + \frac{n}2\log (e^z-1) -
(m+1)\log z \big)dz,
\end{align*}
and the saddle point equation simplifies to
\[
\frac{z(2e^{z}-1)}{e^{z}-1}=1+\frac{2\kappa_n}{n}.
\]
Note that for $n\to\infty$ we have $\frac{\kappa_n}{n}\to 0$;
$\frac{zG'(z)}{G(z)}$ can be expressed in terms of the Bernoulli numbers, such that
\[
\frac{z(2e^{z}-1)}{e^{z}-1} = \sum_{k\ge 0}B_k\big(2\cdot (-1)^k-1\big)\frac{z^k}k! = 1+ \frac32
z+ \frac1{12}z^2-\frac1{720}z^4+\mathcal{O}(z^6),
\]
in a neighbourhood of zero. Thus, we obtain the solution $s_n$ of the saddle point equation,
with $\lim_{n\to\infty}s_n=0$, by a bootstrapping procedure. First, we obtain
\[
s_n=\frac{4}{3}\cdot\frac{\kappa_n}{n}+\mathcal{O}\left(\frac{\kappa_n^2}{n^2}\right).\]
A second bootstrapping step gives the refinement
\[
s_n=\frac{4}{3}\cdot\frac{\kappa_n}{n}-\frac{8}{81}\cdot\frac{\kappa_n^2}{n^2}+\mathcal{O}\left(\frac{\kappa_n^3}{n^3}\right).
\]
Changing the integration path to $z=s_n\cdot e^{i\varphi}$, $-\pi\le \varphi <\pi$ gives for $g=2$
\begin{align*}
\N(f)&= \frac{m!4^m}{2\pi i } \oint_r \exp\big( \frac{n}{2}z + \frac{n}2\log
(e^z-1) - (m+1)\log z \big)dz\\
&=\frac{m!4^m}{2\pi } \int_{-\pi}^{\pi} s_n \exp\big(i\varphi+ \frac{n}{2}\log
\tilde{G}(s_n\cdot e^{i\varphi}) - (m+1)(\log s_n + i\varphi) \big)d\varphi\\
\end{align*}
Since $\tilde{G}(z)=e^{z}\cdot (e^{z}-1)$ we obtain further
\begin{align*}
\N(f)&=\frac{m!4^m s_n^{-m}}{2\pi } \int_{-\pi}^{\pi}
\exp\big(\frac{n}2s_ne^{i\varphi}+\frac{n}2\log (e^{s_n\cdot
e^{i\varphi}}-1)-mi\varphi\big)d\varphi.
\end{align*}
The function $|\tilde{G}(s_n\cdot e^{i\varphi})|$ is maximal at $\varphi=0$. Thus, we restrict ourselves
to a neighbourhood of zero $\varphi\in(-\theta,\theta)$.
The expansion of the term $\frac{n}2\log \tilde{G}(s_n\cdot e^{i\varphi})-im\varphi$ at $\varphi=0$ gives
\begin{align*}
&\frac12n s_n +\frac12 n\log(e^{s_n}-1)+\varphi  i\left(\frac12 ns_n+\frac12
n\frac{s_ne^{s_n}}{e^{s_n}-1}-m\right)\\
&\quad+ \varphi^2 \frac{n s_n}4
\left(\frac{e^{2s_n}s_n}{(e^{s_n}-1)^2}-\frac{e^{s_n}(s_n+1)}{e^{s_n}-1} -1\right)+\mathcal{O}(ns_n
\varphi^3).
\end{align*}
By definition of $s_n$ as the solution of the saddle point equation the linear term vanishes. We
obtain
\begin{align*}
&\N(f)\sim\frac{4^m(2!)^{\frac{n}{2}} s_n^{-m}e^{\frac{ns_n}2}(e^{s_n}-1)^{\frac{n}2}}{2\pi }\\
&\quad\times\int_{-\theta}^{\theta}\exp\left(\varphi^2 \frac{n s_n}4
\left(\frac{e^{2s_n}s_n}{(e^{s_n}-1)^2}-\frac{e^{s_n}(s_n+1)}{e^{s_n}-1} -1\right)
+ \mathcal{O}(ns_n \varphi^3)\right)d\varphi.
\end{align*}
The expansion of $(e^{s_n}-1)^{\frac{n}2}$ gives
\[
(e^{s_n}-1)^{\frac{n}2}=\exp(\frac{n}2\log(e^{s_n}-1))
=\exp\Big(\frac12 n \log s_n+ \frac14 n s_n+\frac1 {48}n s_n^2+\mathcal{O}(n s_n^4)\Big).
\]
Moreover, using
\[
\frac{n s_n}4 \Big(\frac{e^{2s_n}s_n}{(e^{s_n}-1)^2}-\frac{e^{s_n}(s_n+1)}{e^{s_n}-1}
-1\Big)=-\frac38 s_n n -\frac1{24}s_n^2n+\frac{1}{720}n s_n^4+\mathcal{O}(ns_n^6),
\]
we get for the integral the asymptotic expansion
\[
\int_{-\theta}^{\theta}\exp\Big(\varphi^2 \big(-\frac38 s_n n -\frac1{24}s_n^2n\big)
+ \mathcal{O}(ns_n \varphi^2(s_n^2+\varphi))\Big)d\varphi.
\]
Note that the level of precision of the expansions has to be adapted on the actual growth of
$\kappa_n$, here $\kappa_n=\mathcal{O}(n^{\frac23})$.
In the final step we substitute $\varphi=\vartheta/\sqrt{ns_n}$ and complete the tails:
\begin{align*}
\N(f)&\sim\frac{m!4^m \cdot 2^{\frac n2}s_n^{-m+\frac{n}2}e^{\frac{3ns_n}4+\frac1 {48}n
s_n^2}}{2\pi ns_n}\int_{-\theta\sqrt{ns_n}}^{\theta\sqrt{ns_n}}
\exp\Big(\vartheta^2 \big(-\frac38
+\mathcal{O}(s_n+\frac{\vartheta}{\sqrt{ns_n}})\big)\Big)d\vartheta\\
&\sim\frac{m!2^{2m+\frac n2} s_n^{-m+\frac{n}2}e^{\frac{3ns_n}4+\frac1 {48}n s_n^2}}{2\pi ns_n}
\int_{-\infty}^{\infty}
\exp\Big(\vartheta^2 \big(-\frac38\big)\Big)d\vartheta.
\end{align*}
Finally, we use $\frac{1}{\sqrt{2\pi}}\int_{-\infty}^{\infty}e^{-x^2/2}dx=1$ to obtain the
assertion. 
\end{proof}
\section{Discussion}

We have analysed the probability of Boolean functions generated by random 2-Xor expressions. This is strongly related to the 2-Xor-SAT problem. For people working in SAT-solver design the structure of solutions of satisfiable expressions, which corresponds to the component structure of the associated multigraphs, is also important. 

We derived expressions in terms of coefficients of generating functions for the probability of satisfiability in the critical region ($m\sim \frac n2+\Theta(n^{2/3})$) as well as a general expression for the probability of any function (Theorem~\ref{th:proba-f}). Unfortunately, this expression is too complicated to be used for an asymptotic analysis of general functions. So, we discussed several particular classes of functions: Single block functions are completely analyzed. The asymptotic probability very much depends on the range of the excess. For two block functions, the only missing case is that of two large components which are not proportional in size. All those functions are rather close to \false. Finally, functions on the other edge (close to \true, with many blocks of bounded size) 
were studied and, under some regularity conditions on the block sizes, we were able to get the asymptotic probability. 

Apart from extensions of our results to cover, e.g., the extension of Theorem~\ref{th:multigraphs-large-excess} to the supercritical case, or of Proposition~\ref{th:asympt_proportional} to  a larger number of edges,
what is missing is an asymptotic analysis of functions on the boundaries \true\ and \false\ having a more irregular component structure as well as the study of functions in the intermediate range.

\vskip .5cm
{\bf Acknowledgments.}
We thank Herv\'e Daud\'e and Vlady Ravelomanana for fruitful discussions.

\end{document}